\input xy
\xyoption{all}

\documentclass{amsart}

\usepackage{amsmath}
\usepackage{amscd}
\usepackage{rotating}

\def\11{{1\kern-3.5pt 1}}
\def\mumu{{\mu\kern-4.2pt\mu}}

\def\boxtimes{\setbox0\hbox{$\Box$}\copy0\kern-\wd0\hbox{$\times$}}

\def\Hom{\operatorname{Hom}}
\def\Mod{\sf Mod}
\def\Qcoh{\sf Qcoh}

\newtheorem{lemma}{Lemma}[section]
\newtheorem{proposition}[lemma]{Proposition}
\newtheorem{theorem}[lemma]{Theorem}
\newtheorem{corollary}[lemma]{Corollary}

\theoremstyle{definition}

\newtheorem{definition}[lemma]{\sl Definition}

\theoremstyle{remark}

\begin{document}

\pagenumbering{arabic}

\author{Adam Nyman}
\address{Adam Nyman, Department of Mathematics, University of Washington,Seattle, WA 98195}

\title{The Geometry of Points on Quantum Projectivizations}


\begin{abstract}
Suppose $S$ is an affine, noetherian scheme, $X$ is a separated, noetherian $S$-scheme, $\mathcal{E}$ is a coherent ${\mathcal{O}}_{X}$-bimodule and $\mathcal{I} \subset T(\mathcal{E})$ is a graded ideal.  We study the geometry of the functor $\Gamma_{n}$ of flat families of truncated $\mathcal{B}=T(\mathcal{E})/\mathcal{I}$-point modules of length $n+1$.  We then use the results of our study to show that the point modules over $\mathcal{B}$ are parameterized by the closed points of ${\mathbb{P}}_{X^{2}}(\mathcal{E})$.  When $X={\mathbb{P}}^{1}$, we construct, for any $\mathcal{B}$-point module, a graded ${\mathcal{O}}_{X}-\mathcal{B}$-bimodule resolution.
\end{abstract}
\maketitle

\section{Introduction}
The purpose of this paper is to study the geometry of points on quantum projectivizations and to use the results of our study to parameterize points on a quantum ruled surface.

If $S$ is an affine, noetherian scheme, $X$ is a separated, noetherian $S$-scheme, $\mathcal{E}$ is a coherent ${\mathcal{O}}_{X}$-bimodule and $\mathcal{I} \subset T(\mathcal{E})$ is a graded ideal, the functor $\Gamma_{n}$ of flat families of truncated $T(\mathcal{E})/\mathcal{I}$-point modules of length $n+1$ is representable by a closed subscheme of ${\mathbb{P}}_{X^{2}}({\mathcal{E}}^{\otimes n})$ \cite{8}.  Truncating a truncated family of point modules of length $i+1$ by taking its first $i$ components defines a morphism $\Gamma_{i} \rightarrow \Gamma_{i-1}$ which makes the set $\{\Gamma_{n}\}$ an inverse system.  In order for the point modules of $\mathcal{B}=T(\mathcal{E})/\mathcal{I}$ to be parameterizable by a scheme, this system must be eventually constant.  If $k$ is a field and $S=X= \operatorname{Spec }k$, Artin, Tate and Van den Bergh prove that $\Gamma_{n}$ is representable \cite[Proposition 3.9, p.48]{1} and they describe sufficient conditions for the inverse system $\{\Gamma_{n}\}$ to be eventually constant \cite[Propositions 3.5, 3.6 and 3.7, p.44-45]{1}.  They then show that these conditions are satisfied when $\mathcal{B}$ is a regular algebra of global dimension two or three generated in degree one.

Returning to the case that $S$ is an arbitrary affine, noetherian scheme and $X$ is a separated, noetherian $S$-scheme, we prove, as suggested by Van den Bergh, analogues of \cite[Propositions 3.5, 3.6 and 3.7, p.44-45]{1} in order to give sufficient conditions for the inverse system $\{\Gamma_{n}\}$ to be eventually constant.  We then show that when ${\sf Proj} \mathcal{B}$ is a quantum ruled surface, sufficient conditions for the inverse system $\{\Gamma_{n}\}$ to be eventually constant are satisfied and the point modules over $\mathcal{B}$ are parameterized by a closed subscheme of ${\mathbb{P}}_{X^{2}}(\mathcal{E})$ whose closed points agree with those of ${\mathbb{P}}_{X^{2}}(\mathcal{E})$.  Van den Bergh proves \cite[Proposition 5.3.1]{14} that, when ${\sf Proj} \mathcal{B}$ is a quantum ruled surface, all the $\Gamma_{n}$'s are isomorphic without the use of geometric techniques.  While our parameterization of point modules over $\mathcal{B}$ follows from the work of Van den Bergh, our proof, which is distinct from Van den Bergh's proof, serves to illustrate the utility of the geometry of the $\Gamma_{n}$'s.

Quantum ruled surfaces were first defined by Michel Van den Bergh as follows:  Suppose $X$ is a smooth curve over an algebraically closed field $k$ and $\mathcal{E}$ is an ${\mathcal{O}}_{X}$-bimodule which is locally free of rank two.  If $\mathcal{Q} \subset \mathcal{E} \otimes_{{\mathcal{O}}_{X}} \mathcal{E}$ is an invertible bimodule, then the quotient $\mathcal{B} = T(\mathcal{E})/(\mathcal{Q})$ is a bimodule algebra.  A quantum ruled surface is the quotient of ${\sf{Grmod }}\mathcal{B}$ by direct limits of modules which are zero in high degree \cite[p.33]{9}.  In order that $\mathcal{B}$ has desired regularity properties, Patrick imposes the condition of admissibility on $\mathcal{Q}$ \cite[Section 2.3]{9}.  We take a different approach by insisting only that $\mathcal{Q}$ be nondegenerate (Definition \ref{def.nondeg}).

Van den Bergh has developed another definition of quantum ruled surface \cite[Definition 11.4, p.35]{12} based on the notion of a non-commutative symmetric algebra generated by $\mathcal{E}$, which does not depend on $\mathcal{Q}$.  He shows that the point modules over a non-commutative symmetric algebra are parameterized by ${\mathbb{P}}_{X^{2}}(\mathcal{E})$ then uses this parameterization to show that the category of graded modules over a non-commutative symmetric algebra is noetherian.

The paper is organized as follows.  In Section 2 we recall the definition of the functor $\Gamma_{n}$, and describe its representing scheme (\cite[Theorem 7.1]{8}).  We then prove analogues of \cite[Propositions 3.5, 3.6 and 3.7, p.44-45]{1}(Propositions \ref{prop.geo}, \ref{prop.prelim} and \ref{prop.atv}) which provide us with sufficient conditions for the inverse system $\{\Gamma_{n}\}$ to be eventually constant.  In Section 3 we review a notion of quantum ruled surface due to Van den Bergh.  The definition of a quantum ruled surface is given in terms of the dual of a locally free ${\mathcal{O}}_{X}$-bimodule, and we show that duality extends to a functor $(-)^{*}:{\sf bimod }X \rightarrow {\sf bimod }X$.  In Section 4, we use the results of Section 2 to show that if ${\sf Proj} \mathcal{B}$ is a quantum ruled surface, the point modules over $\mathcal{B}$ are parameterized by the closed points of ${\mathbb{P}}_{X^{2}}(\mathcal{E})$.  In this case we show that, if $X={\mathbb{P}}^{1}$, the point modules over $\mathcal{B}$ have the expected resolution.  More precisely, we have the following Proposition (Proposition \ref{lem.hseries}):

\begin{proposition}
If ${\sf Proj} \mathcal{B}$ is a quantum ruled surface, an object $\mathcal{M}$ in ${\sf{Grmod }} \mathcal{B}$ with multiplication map $\rho:{\mathcal{M}}_{0}\otimes \mathcal{B} \rightarrow \mathcal{M}$ and isomorphism $\phi:{\mathcal{O}}_{p} \rightarrow {\mathcal{M}}_{0}$ for $p$ a closed point in $X$ has a graded ${\mathcal{O}}_{X}-{\mathcal{B}}$-bimodule resolution
$$
\xymatrix{
0 \ar[r] & ({\mathcal{O}}_{q} \otimes_{{\mathcal{O}}_{X}} \mathcal{B})(-1) \ar[r] & {\mathcal{O}}_{p} \otimes_{{\mathcal{O}}_{X}} \mathcal{B} \ar[rr]^{\rho (\phi \otimes \mathcal{B})} & & {\mathcal{M}} \ar[r] & 0
}
$$
for $q$ a closed point in $X$ if and only if $\mathcal{M}$ is a point module.
\end{proposition}

{\bf Acknowledgments.}  We thank D. Patrick, S.P. Smith and M. Van den Bergh for numerous enlightening conversations.  We thank S.P. Smith for, among many other things, providing a proof of Proposition \ref{prop.smith}.  We thank M. Van den Bergh for inviting the author to visit him at Limburgs Universitair Centrum, for providing our definition of quantum ruled surface and for suggesting a strategy of proof for Proposition \ref{theorem.converge}.

{\bf Notation and Conventions}
Throughout, assume $k$ is an algebraically closed field.

If $\sf A$ and $\sf B$ are categories and $F: {\sf A} \rightarrow {\sf B}$ and $G: {\sf B} \rightarrow {\sf A}$ are functors we write $(F,G)$ if $F$ is left adjoint to $G$.

If $X$ is a scheme, let ${\sf Qcoh} X$ denote the category of quasi-coherent ${\mathcal{O}}_{X}$-modules.  We say $\mathcal{M}$ is an ${\mathcal{O}}_{X}$-module if $\mathcal{M}$ is a quasi-coherent ${\mathcal{O}}_{X}$-module.

\section{The Geometry of Quantum Projectivizations}
In this section we parameterize point modules over a bimodule algebra (\cite[Definition 2.3, p.440]{15}) generated in degree one.
\subsection{Parameterizations of truncated point modules over bimodule algebras}
\begin{definition}
Suppose $X$ is a noetherian $k$-scheme, and ${\mathcal{B}}$ is a graded ${{\mathcal{O}}_{X}}$-bimodule algebra.  A {\bf point module} over $\mathcal{B}$ is an ${\mathbb{N}}$-graded $\mathcal{B}$-module ${\mathcal{M}}_{0} \oplus {\mathcal{M}}_{1} \oplus \cdots $ such that, for each $i\geq 0$, the multiplication map ${\mathcal{M}}_{i} \otimes_{{\mathcal{O}}_{X}} {\mathcal{B}}_{1} \rightarrow {\mathcal{M}}_{i+1}$ is epic and ${\mathcal{M}}_{i} \cong {\mathcal{O}}_{p_{i}}$ for some closed point $p_{i} \in X$.
\end{definition}
Fix a noetherian affine scheme $S$ and suppose $X$ is a separated noetherian $S$-scheme.  If $\mathcal{B}$ is a graded ${\mathcal{O}}_{X}$-bimodule algebra generated in degree 1, $U$ is a noetherian affine $S$-scheme, $p:(U \times X) \times_{U} (U \times X) \rightarrow X^{2}$ is projection then ${\mathcal{B}}^{U} = p^{*}\mathcal{B}$ has the structure of an ${\mathcal{O}}_{U \times X}$-bimodule algebra \cite[Proposition 3.42, p.47]{8}.

\begin{definition}
Let $\mathcal{B}$ be a graded ${{\mathcal{O}}_{X}}$-bimodule algebra and let $U$ be an affine scheme.  A {\bf family of $\mathcal{B}$-point modules parameterized by $\mathbf{U}$} or a {\bf $\mathbf{U}$-family over $\mathcal{B}$} is a graded ${\mathcal{B}}^{U}$-module ${\mathcal{M}}_{0} \oplus {\mathcal{M}}_{1} \oplus \cdots $ such that for each $i\geq 0$, the multiplication map ${\mathcal{M}}_{i} \otimes_{{\mathcal{O}}_{X}} {\mathcal{B}}_{1} \rightarrow {\mathcal{M}}_{i+1}$ is epic and such that there exists a map
$$
q_{i}: U \rightarrow X
$$
and an invertible ${\mathcal{O}}_{U}$-module ${\mathcal{L}}_{i}$ with ${\mathcal{L}}_{0} \cong {\mathcal{O}}_{U}$ and
$$
{\mathcal{M}}_{i} \cong (id_{U} \times q_{i})_{*} {\mathcal{L}}_{i}.
$$
A {\bf family of truncated $\mathcal{B}$-point modules of length $\mathbf{n}$ parameterized by $\mathbf{U}$}, or a {\bf truncated $\mathbf{U}$-family of length $\mathbf{n}$} is  a graded ${\mathcal{B}}^{U}$-module ${\mathcal{M}}_{0} \oplus {\mathcal{M}}_{1} \oplus \cdots $such that for each $i\geq 0$, the multiplication map ${\mathcal{M}}_{i} \otimes_{{\mathcal{O}}_{X}} {\mathcal{B}}_{1} \rightarrow {\mathcal{M}}_{i+1}$ is epic and such that there exists a map
$$
q_{i}: U \rightarrow X
$$
and an invertible ${\mathcal{O}}_{U}$-module ${\mathcal{L}}_{i}$ with ${\mathcal{L}}_{0} \cong {\mathcal{O}}_{U}$,
$$
{\mathcal{M}}_{i} \cong (id_{U} \times q_{i})_{*} {\mathcal{L}}_{i}
$$
for $i \leq n$, and ${\mathcal{M}}_{i}=0$ for $i > n$.
\end{definition}
We note that a point module over $\mathcal{B}$ is a family of $\mathcal{B}$-point modules parameterized by $\operatorname{Spec }k$.  If $f:V \rightarrow U$ is a map of noetherian affine $S$-schemes, we let $\tilde{f}:V \times X \rightarrow U \times X$ be the map $f \times \operatorname{id}_{X}$.

\begin{definition}
Let $\sf{S}$ be the category of affine, noetherian $S$-schemes.  The assignment $\Gamma_{n}:{\sf{S}} \rightarrow {\sf Sets}$ sending $U$ to
$$
\{ \mbox{isomorphism classes of truncated $U$-families of length n+1} \}
$$
and sending $f:V \rightarrow U$ to the map $\Gamma_{n}(f)$ defined by $\Gamma_{n}(f)[\mathcal{M}] = [{\tilde{f}}^{*}\mathcal{M}]$, is the {\bf functor of flat families of truncated $\mathcal{B}$-point modules of length $\mathbf{n+1}$}.
\end{definition}
We now describe the space which parameterizes truncated point modules.  We begin with some notation.  If $X$, $Y$ and $Z$ are $S$-schemes, $U$ is an $X\times Y$-scheme with structure map $u$, $W$ is a $Y \times Z$-scheme with structure map $w$, and if $pr_{i}:X \times Y \rightarrow X,Y$ and $pr_{i}':Y\times Z \rightarrow Y,Z$ are projections, we define the tensor product of $U$ and $W$ over $Y$, denoted $U\otimes_{Y}W$, to be the pullback of
$$
\xymatrix{
& W \ar[d]^{pr_{1}'w} \\
U \ar[r]_{pr_{2}u} & Y.
}
$$
Now suppose $S$ is a noetherian affine scheme, and $X$, $Y$, $Z$ are noetherian separated $S$-schemes.

\begin{theorem} \cite[Theorem 6.3, p.93-94]{8} \label{theorem.segre}
If $\mathcal{E}$ is a coherent ${\mathcal{O}}_{X}- {\mathcal{O}}_{Y}$-bimodule and $\mathcal{F}$ is a coherent ${\mathcal{O}}_{Y}-{\mathcal{O}}_{Z}$-bimodule then there exists a canonical map
$$
s:{\mathbb{P}}_{X\times Y}(\mathcal{E}) \otimes_{Y} {\mathbb{P}}_{Y \times Z}(\mathcal{F}) \rightarrow {\mathbb{P}}_{X \times Z}({\mathcal{E}} \otimes_{{\mathcal{O}}_{Y}} {\mathcal{F}})
$$
such that $s$ is a closed immersion which is functorial, associative, and compatible with base change.  We call $s$ the {\bf bimodule Segre embedding}.
\end{theorem}
For the rest of this section, suppose $\mathcal{E}$ is a coherent ${\mathcal{O}}_{X}$-bimodule, $\mathcal{I} \subset T(\mathcal{E})$ is an ideal, and $\mathcal{B} = T(\mathcal{E})/\mathcal{I}$.
Define the trivial bimodule Segre embedding $s:{\mathbb{P}}_{X^{2}}(\mathcal{E})^{\otimes 1} \rightarrow {\mathbb{P}}_{X^{2}}({\mathcal{E}}^{\otimes 1})$ as the identity map.
\begin{theorem} \cite[Theorem 7.1, p. 118]{8}
For $n \geq 1$, $\Gamma_{n}$ is represented by the pullback of the diagram
\begin{equation} \label{eqn.gamma}
\xymatrix{
& {\mathbb{P}}_{X^{2}}(\mathcal{E})^{\otimes n} \ar[d]^{s} \\
{\mathbb{P}}_{X^{2}}({\mathcal{E}}^{\otimes n}/{\mathcal{I}}_{n}) \ar[r] & {\mathbb{P}}_{X^{2}}({\mathcal{E}}^{\otimes n})
}
\end{equation}
\end{theorem}
We will abuse notation by calling the pullback of (\ref{eqn.gamma}) $\Gamma_{n}$.

\subsection{Sufficient conditions for system $\{\Gamma_{n}\}$ to be eventually constant}
The proof of the following Lemma is straightforward so we omit it.
\begin{lemma} \label{lem.closedimm}
Let
$$
\xymatrix{
A \times_{C} B \ar[r]^{a} \ar[d]_{b} & A \ar[d]^{c} \\
B \ar[r]_{d} & C
}
$$
be a pullback diagram of schemes.  If $c$ and $d$ are closed immersions, then so are $a$ and $b$.  Furthermore, if $f:D \rightarrow A$ and $g:D \rightarrow B$ induce a morphism $h:D \rightarrow A \times_{C} B$, and either $f$ or $g$ is a closed immersion, then so is $h$.
\end{lemma}
If $\mathcal{J} \subset {\mathcal{E}}^{\otimes d}$ is a submodule, define ${\mathcal{Z}}(\mathcal{J})$ as the pullback of the diagram
$$
\xymatrix{
& {\mathbb{P}}_{X^{2}}(\mathcal{E})^{\otimes n} \ar[d]^{s} \\
{\mathbb{P}}_{X^{2}}({\mathcal{E}}^{\otimes n}/{\mathcal{J}}) \ar[r] & {\mathbb{P}}_{X^{2}}({\mathcal{E}}^{\otimes n}).
}
$$
For the readers convenience, we recall Grothendieck's description of maps to projective bundles along with a result from \cite{8} which we will need below.

\begin{proposition} \label{prop.groth} \cite[Proposition 4.2.3, p.73]{4}.
Let $q: U \rightarrow V$ be a morphism of schemes.  Then, given an ${\mathcal{O}}_{V}$-module $\mathcal{G}$ there is a bijective correspondence between the set of $V$-morphisms $r:U \rightarrow {\mathbb{P}}_{V}(\mathcal{G})$, and the set of equivalence classes of pairs $(\mathcal{L}, \phi)$ composed of an invertible ${\mathcal{O}}_{U}$-module $\mathcal{L}$ and a epimorphism $\phi:q^{*}(\mathcal{G}) \rightarrow \mathcal{L}$, where two pairs $(\mathcal{L}, \phi)$ and $({\mathcal{L}}', \phi ')$ are equivalent if there exists an ${\mathcal{O}}_{U}$-module isomorphism $\tau: \mathcal{L} \rightarrow {\mathcal{L}}'$ such that

$$
\xymatrix
{
q^{*}(\mathcal{G}) \ar[r]^{\phi} \ar[dr]_{\phi '} & \mathcal{L}  \ar[d]^{\tau} \\
& \mathcal{L} '
}
$$
commutes.
\end{proposition}
If $q:U \rightarrow V$ is a map of noetherian schemes, $\mathcal{G}$ is an ${\mathcal{O}}_{V}$-module and $\mathcal{L}$ is an ${\mathcal{O}}_{U}$-module, any ${\mathcal{O}}_{U}$-module morphism $\phi:q^{*}\mathcal{G} \rightarrow \mathcal{L}$ corresponds to an ${\mathcal{O}}_{V}$-module morphism $\psi:\mathcal{G} \rightarrow q_{*}\mathcal{L}$ since $(g^{*},g_{*})$ is an adjoint pair.  We say $\phi$ is the left adjunct of $\psi$.  We will use this correspondence implicitly.

Choose $n \geq 1$, suppose $U$ is an $S$-scheme, for $0 \leq i \leq n$ $q_{i}:U \rightarrow X$ is a morphism and for $1 \leq i \leq n$, ${\mathcal{L}}_{i}$ is an invertible ${\mathcal{O}}_{U}$-module.  We recall the existence of a natural morphism
$$
\xymatrix{
\otimes_{i=0}^{n-1} (q_{i}\times q_{i+1})_{*}{\mathcal{L}}_{i+1} \ar[r]^{\gamma} & (q_{0} \times q_{n})_{*}(\otimes_{i=0}^{n-1}{\mathcal{L}}_{i+1})
}
$$
whose domain is the bimodule tensor product of ${\mathcal{O}}_{X^{2}}$-modules and whose codomain is the direct image of the ordinary tensor product of ${\mathcal{O}}_{U}$-modules.
\begin{proposition} \cite[Theorem 6.3, p.93-94, Lemma 7.6, p.124]{8} \label{prop.pain}
Retain the notation above.  If $\mathcal{J} \subset {\mathcal{E}}^{\otimes n}$ is a submodule and $U$ is an $S$-scheme, then $X^{2}$-morphisms $f:U \rightarrow {\mathcal{Z}}(\mathcal{J})$ correspond, via Proposition \ref{prop.groth}, to $n$-tuples of ${\mathcal{O}}_{X^{2}}$-module maps $\psi_{i}:\mathcal{E} \rightarrow (q_{i}\times q_{i+1})_{*}{\mathcal{L}}_{i+1}$ such that the left adjunct of $\psi_{i}$ is an epi and such that $\mathcal{J}$ is in the kernel of the composition
\begin{equation} \label{eqn.star}
\xymatrix{
{\mathcal{E}}^{\otimes n} \ar[rrr]^{\otimes_{i=0}^{n-1}\psi_{i}} & & & \otimes_{i=0}^{n-1} (q_{i}\times q_{i+1})_{*}{\mathcal{L}}_{i+1} \ar[r]^{\gamma} & (q_{0} \times q_{n})_{*}(\otimes_{i=0}^{n-1}{\mathcal{L}}_{i+1}).
}
\end{equation}
\end{proposition}
We define ${\mathcal{Z}}(\mathcal{E}\otimes {\mathcal{I}}_{d}) \cap {\mathcal{Z}}({\mathcal{I}}_{d} \otimes \mathcal{E})$ as the pullback of
$$
\xymatrix{
& {\mathcal{Z}}(\mathcal{E}\otimes {\mathcal{I}}_{d}) \ar[d] \\
{\mathcal{Z}}({\mathcal{I}}_{d} \otimes \mathcal{E}) \ar[r] & {\mathbb{P}}_{X^{2}}({\mathcal{E}}^{\otimes d+1}).
}
$$
\begin{proposition} \label{prop.geo}
For any $d$, there exists a closed immersion
$$
\Gamma_{d+1} \rightarrow {\mathcal{Z}}(\mathcal{E}\otimes {\mathcal{I}}_{d}) \cap {\mathcal{Z}}({\mathcal{I}}_{d} \otimes \mathcal{E})
$$
over ${{\mathbb{P}}_{X^{2}}(\mathcal{E})}^{\otimes d+1}$ which is an isomorphism if ${\mathcal{I}}_{d+1}=\mathcal{E} \otimes {\mathcal{I}}_{d}+ {\mathcal{I}}_{d} \otimes \mathcal{E}$.
\end{proposition}

\begin{proof}
First, note that the inclusion $\mathcal{E} \otimes {\mathcal{I}}_{d} \subset {\mathcal{I}}_{d+1}$ induces a closed immersion
\begin{equation} \label{eqn.closed1}
{\mathbb{P}}_{X^{2}}({\mathcal{E}}^{\otimes d+1}/{\mathcal{I}}_{d+1}) \rightarrow {\mathbb{P}}_{X^{2}}({\mathcal{E}}^{\otimes d+1}/\mathcal{E} \otimes {\mathcal{I}}_{d}).
\end{equation}
Similarly, there is a closed immersion
\begin{equation} \label{eqn.closed2}
{\mathbb{P}}_{X^{2}}({\mathcal{E}}^{\otimes d+1}/{\mathcal{I}}_{d+1}) \rightarrow {\mathbb{P}}_{X^{2}}({\mathcal{E}}^{\otimes d+1}/{\mathcal{I}}_{d} \otimes \mathcal{E})
\end{equation}
such that (\ref{eqn.closed1}) and (\ref{eqn.closed2}) make the diagram
$$
\xymatrix{
{\mathbb{P}}_{X^{2}}({\mathcal{E}}^{\otimes d+1}/{\mathcal{I}}_{d} \otimes \mathcal{E}) \ar[d] &  {\mathbb{P}}_{X^{2}}({\mathcal{E}}^{\otimes d+1}/{\mathcal{I}}_{d+1})   \ar[l] \ar[d]^{=} \\
{\mathbb{P}}_{X^{2}}({\mathcal{E}}^{\otimes d+1}) & {\mathbb{P}}_{X^{2}}({\mathcal{E}}^{\otimes d+1}/{\mathcal{I}}_{d+1}) \ar[l] \ar[d]^{=} \\
{\mathbb{P}}_{X^{2}}({\mathcal{E}}^{\otimes d+1}/\mathcal{E} \otimes {\mathcal{I}}_{d}) \ar[u] &  {\mathbb{P}}_{X^{2}}({\mathcal{E}}^{\otimes d+1}/{\mathcal{I}}_{d+1}) \ar[l]
}
$$
commute.  Thus, the diagram of closed immersions
$$
\xymatrix{
{\mathcal{Z}}({\mathcal{I}}_{d+1}) \ar[d] \ar[r] &  {\mathbb{P}}_{X^{2}}({\mathcal{E}}^{\otimes d+1}/{\mathcal{I}}_{d+1}) \ar[r] & {\mathbb{P}}_{X^{2}}({\mathcal{E}}^{\otimes d+1}/\mathcal{E} \otimes {\mathcal{I}}_{d}) \ar[d] \\
{\mathbb{P}}_{X^{2}}(\mathcal{E})^{\otimes d+1} \ar[rr] & & {\mathbb{P}}_{X^{2}}({\mathcal{E}}^{\otimes d+1})
}
$$
induces, by the universal property of the pullback, a morphism $\mathcal{Z}({\mathcal{I}}_{d+1}) \rightarrow \mathcal{Z}(\mathcal{E} \otimes {\mathcal{I}}_{d})$.  By Lemma \ref{lem.closedimm} this morphism is a closed immersion.  In a similar fashion, we obtain a closed immersion $\mathcal{Z}({\mathcal{I}}_{d+1}) \rightarrow  \mathcal{Z}({\mathcal{I}}_{d} \otimes \mathcal{E})$.  By construction, these closed immersions are morphisms over ${\mathbb{P}}_{X^{2}}(\mathcal{E})^{\otimes d+1}$.  Hence, they are morphisms over ${\mathbb{P}}_{X^{2}}({\mathcal{E}}^{\otimes d+1})$ so they induce a morphism
\begin{equation} \label{eqn.closed}
\mathcal{Z}({\mathcal{I}}_{d+1}) \rightarrow \mathcal{Z}(\mathcal{E} \otimes {\mathcal{I}}_{d}) \cap \mathcal{Z}({\mathcal{I}}_{d} \otimes \mathcal{E})
\end{equation}
which is a closed immersion by Lemma \ref{lem.closedimm}.  We next show that, if ${\mathcal{I}}_{d+1}=\mathcal{E} \otimes {\mathcal{I}}_{d}+ {\mathcal{I}}_{d} \otimes \mathcal{E}$,
$$
\mathcal{Z}(\mathcal{E} \otimes {\mathcal{I}}_{d}) \cap \mathcal{Z}({\mathcal{I}}_{d} \otimes \mathcal{E})=\mathcal{Z}(\mathcal{E} \otimes {\mathcal{I}}_{d} + {\mathcal{I}}_{d} \otimes \mathcal{E}).
$$
We will show that the functors
$$
\operatorname{Hom}_{S}(-,\mathcal{Z}(\mathcal{E} \otimes {\mathcal{I}}_{d}) \cap \mathcal{Z}({\mathcal{I}}_{d} \otimes \mathcal{E}))
$$
and
$$
\operatorname{Hom}_{S}(-,\mathcal{Z}(\mathcal{E} \otimes {\mathcal{I}}_{d} + {\mathcal{I}}_{d} \otimes \mathcal{E}))
$$
are isomorphic, which, by Yoneda's Lemma, will prove the assertion.  To this end, suppose $U$ is a noetherian affine $S$-scheme and $f:U \rightarrow \mathcal{Z}(\mathcal{E} \otimes {\mathcal{I}}_{d}) \cap \mathcal{Z}({\mathcal{I}}_{d} \otimes \mathcal{E})$ is a map of $S$-schemes.  By Proposition \ref{prop.pain}, $f$ corresponds to a map of ${\mathcal{O}}_{X^{2}}$-modules of the form (\ref{eqn.star}) such that both $\mathcal{E} \otimes {\mathcal{I}}_{d}$ and ${\mathcal{I}}_{d} \otimes \mathcal{E}$ are contained in the kernel of $\phi$.  Thus, the sum $\mathcal{E} \otimes {\mathcal{I}}_{d} + {\mathcal{I}}_{d} \otimes \mathcal{E}$ is contained in the kernel of $\phi$.  By Proposition \ref{prop.pain} again, there is a morphism $h:U \rightarrow \mathcal{Z}(\mathcal{E} \otimes {\mathcal{I}}_{d} + {\mathcal{I}}_{d} \otimes \mathcal{E})$ corresponding to $\phi$ making the diagram
\begin{equation} \label{eqn.comm}
\xymatrix{
U \ar[rrr]^{f} \ar[drrr]_{h} & & & \mathcal{Z}(\mathcal{E} \otimes {\mathcal{I}}_{d}) \cap \mathcal{Z}({\mathcal{I}}_{d} \otimes \mathcal{E}) \\
& & & \mathcal{Z}(\mathcal{E} \otimes {\mathcal{I}}_{d} + {\mathcal{I}}_{d} \otimes \mathcal{E}) \ar[u]
}
\end{equation}
whose right vertical is (\ref{eqn.closed}), commute.  Since (\ref{eqn.closed}) is a closed immersion, $h$ is unique with this property.  The assignment sending $f$ to $h$ gives a bijection
$$
\operatorname{Hom}_{S}(U,\mathcal{Z}(\mathcal{E} \otimes {\mathcal{I}}_{d}) \cap \mathcal{Z}({\mathcal{I}}_{d} \otimes \mathcal{E})) \rightarrow \operatorname{Hom}_{S}(U,\mathcal{Z}(\mathcal{E} \otimes {\mathcal{I}}_{d} + {\mathcal{I}}_{d} \otimes \mathcal{E})).
$$
We now show this bijection is natural.  Suppose $r:V \rightarrow U$ and let $f:U \rightarrow \mathcal{Z}(\mathcal{E} \otimes {\mathcal{I}}_{d}) \cap \mathcal{Z}({\mathcal{I}}_{d} \otimes \mathcal{E})$ be given as above.  Then, as above, the composition $fr$ corresponds to a map $h':V \rightarrow \mathcal{Z}(\mathcal{E} \otimes {\mathcal{I}}_{d} + {\mathcal{I}}_{d} \otimes \mathcal{E})$ which is unique making the diagram
$$
\xymatrix{
U \ar[rrr]^{f} & & & \mathcal{Z}(\mathcal{E} \otimes {\mathcal{I}}_{d}) \cap \mathcal{Z}({\mathcal{I}}_{d} \otimes \mathcal{E}) \\
V \ar[u]^{r} \ar[rrr]_{h'} & & & \mathcal{Z}(\mathcal{E} \otimes {\mathcal{I}}_{d} + {\mathcal{I}}_{d} \otimes \mathcal{E}) \ar[u]
}
$$
commute.  Since (\ref{eqn.comm}) commutes, the uniqueness of $h'$ implies that $hr=h'$ as desired.
\end{proof}
Note that there exists a closed immersion $\Gamma_{r} \otimes_{X} {\mathbb{P}}_{X^{2}}(\mathcal{E}) \rightarrow \mathcal{Z}({\mathcal{I}}_{r} \otimes \mathcal{E})$ induced by the commutative diagram of closed immersions
$$
\xymatrix{
\Gamma_{r} \otimes_{X} {\mathbb{P}}_{X^{2}}(\mathcal{E}) \ar[r] \ar[d] & {\mathbb{P}}_{X^{2}}(\mathcal{E})^{\otimes r+1} \ar[d] \\
{\mathbb{P}}_{X^{2}}({\mathcal{E}}^{\otimes r}/{\mathcal{I}}_{r}) \otimes_{X} {\mathbb{P}}_{X^{2}}(\mathcal{E}) \ar[d] \ar[r] &  {\mathbb{P}}_{X^{2}}({\mathcal{E}}^{\otimes r}) \otimes_{X} {\mathbb{P}}_{X^{2}}(\mathcal{E}) \ar[d] \\
{\mathbb{P}}_{X^{2}}(({\mathcal{E}}^{\otimes r}/{\mathcal{I}}_{r}) \otimes \mathcal{E}) \ar[r] &  {\mathbb{P}}_{X^{2}}({\mathcal{E}}^{\otimes r+1}).
}
$$
In a similar fashion, there exists a closed immersion ${\mathbb{P}}_{X^{2}}(\mathcal{E}) \otimes_{X} \Gamma_{r} \rightarrow \mathcal{Z}(\mathcal{E} \otimes {\mathcal{I}}_{r})$.  Define $\Gamma_{r} \otimes_{X} {\mathbb{P}}_{X^{2}}(\mathcal{E}) \cap  {\mathbb{P}}_{X^{2}}(\mathcal{E}) \otimes_{X} \Gamma_{r}$ as the pullback of the diagram
\begin{equation} \label{eqn.pullbacku}
\xymatrix{
& & \Gamma_{r} \otimes_{X} {\mathbb{P}}_{X^{2}}(\mathcal{E}) \ar[d] \\
& &  \mathcal{Z}({\mathcal{I}}_{r} \otimes \mathcal{E}) \ar[d] \\
{\mathbb{P}}_{X^{2}}(\mathcal{E}) \otimes_{X} \Gamma_{r} \ar[r] & \mathcal{Z}(\mathcal{E} \otimes {\mathcal{I}}_{r}) \ar[r] & {\mathbb{P}}_{X^{2}}({\mathcal{E}}^{\otimes r+1}).
}
\end{equation}

\begin{corollary} \label{cor.include}
If $\mathcal{I}$ is generated in degrees $\leq d$, then there is a closed immersion
$$
\Gamma_{r} \otimes_{X} {\mathbb{P}}_{X^{2}}(\mathcal{E}) \cap  {\mathbb{P}}_{X^{2}}(\mathcal{E}) \otimes_{X} \Gamma_{r} \rightarrow \Gamma_{r+1}
$$
over ${\mathbb{P}}_{X^{2}}(\mathcal{E})^{\otimes r+1}$ for $r \geq d$.
\end{corollary}

\begin{proof}
It is clear from (\ref{eqn.pullbacku}) that there is a morphism
$$
\Gamma_{r} \otimes_{X} {\mathbb{P}}_{X^{2}}(\mathcal{E}) \cap  {\mathbb{P}}_{X^{2}}(\mathcal{E}) \otimes_{X} \Gamma_{r} \rightarrow  \mathcal{Z}({\mathcal{I}}_{r} \otimes \mathcal{E}) \cap \mathcal{Z}(\mathcal{E} \otimes {\mathcal{I}}_{r})
$$
which is a closed immersion by Lemma \ref{lem.closedimm}.  Furthermore, it is not hard to show that this morphism is over ${\mathbb{P}}_{X^{2}}(\mathcal{E})^{\otimes r+1}$.  The assertion follows from Proposition \ref{prop.geo}.
\end{proof}

\begin{lemma} \label{lem.utilize}
Suppose $m,n$ are integers such that $1 \leq m < n \leq d$.  If $pr_{m,n}^{d}$ denotes the projection ${\mathbb{P}}_{X^{2}}(\mathcal{E})^{\otimes d} \rightarrow {\mathbb{P}}_{X^{2}}(\mathcal{E})^{\otimes n-m+1}$ onto the $m$th through the $n$th factors of ${\mathbb{P}}_{X^{2}}(\mathcal{E})^{\otimes d}$, then the composition
$$
\xymatrix{
\Gamma_{d} \ar[r] & {\mathbb{P}}_{X^{2}}(\mathcal{E})^{\otimes d} \ar[r]^{pr_{m,n}^{d}} & {\mathbb{P}}_{X^{2}}(\mathcal{E})^{\otimes n-m+1}
}
$$
factors through the map $\Gamma_{n-m+1} \rightarrow {\mathbb{P}}_{X^{2}}(\mathcal{E})^{\otimes n-m+1}$.  Furthermore, the induced map, which we call $p_{m,n}^{d}:\Gamma_{d} \rightarrow \Gamma_{n-m+1}$ is closed.
\end{lemma}

\begin{proof}
Let $U$ be an affine, noetherian $S$-scheme.  If $\oplus_{i=0}^{d}{\mathcal{M}}_{i}$ is a family of $\mathcal{B}$-point modules of length $d+1$ parameterized by $U$,  then $\oplus_{i=m-1}^{n}{\mathcal{M}}_{i}$ is a family of $\mathcal{B}$-point modules of length $n-m+2$ parameterized by $U$.  This assignment induces a natural transformation of functors $\Gamma_{d} \rightarrow \Gamma_{n-m+1}$ so that by \cite[Theorem 7.1, p.118]{8} and Yoneda's lemma, the assignment induces the map of schemes $p_{m,n}^{d}:\Gamma_{d} \rightarrow \Gamma_{n-m+1}$.  The first assertion follows.  The second assertion will follow from the fact that $p_{m,n}^{d}$ is projective.  This is the content of the following result.
\end{proof}

\begin{lemma} \label{lem.proper}
The map $p_{m,n}^{d}$ is projective.
\end{lemma}

\begin{proof}
Since closed immersions are projective and since a product of projective maps is projective by \cite[5.5.5.(i),(iv), p.105]{4}, the result follows from the fact that ${\mathbb{P}}_{X^{2}}(\mathcal{E})$ is projective over $X^{2}$.
\end{proof}

\begin{lemma} \label{lemma.closed}
Suppose $A$ and $B$ are schemes, $A$ is noetherian and
$$
\xymatrix{
A \ar[r]^{f} & B \ar[r]^{g} & A
}
$$
is a diagram of closed immersions.  If $fg=id_{B}$, then $gf=id_{A}$.
\end{lemma}

\begin{proof}
Since $f$ and $g$ are closed immersions, the sheaf maps $f^{\#}:{\mathcal{O}}_{A} \rightarrow f_{*}{\mathcal{O}}_{B}$ and $g^{\#}:{\mathcal{O}}_{B} \rightarrow g_{*}{\mathcal{O}}_{A}$ corresponding to $f$ and $g$ are surjections.  By hypothesis, $(fg)_{*}=id_{B*}$ so that the composition
$$
\xymatrix{
{\mathcal{O}}_{A} \ar[r]^{f^{\#}} & f_{*}{\mathcal{O}}_{B} \ar[r]^{f_{*}g^{\#}} & f_{*}g_{*}{\mathcal{O}}_{A}={\mathcal{O}}_{A}
}
$$
is a surjection of sheaves of rings.  Since ${\mathcal{O}}_{A}$ is noetherian, the above composition is an isomorphism.  Next, consider the composition
$$
\xymatrix{
{\mathcal{O}}_{A} \ar[r]^{f^{\#}} & f_{*}{\mathcal{O}}_{B} \ar[r]^{f_{*}g^{\#}} & f_{*}g_{*}{\mathcal{O}}_{A} \ar[r]^{f^{\#}} & g_{*}{\mathcal{O}}_{B}
}
$$
restricted to an affine open set $U \subset A$.  Let $a \in {\mathcal{O}}_{A}(U)$.  Then $f^{\#}(a-f_{*}g^{\#}f^{\#}(a))=0$ so that $f_{*}g^{\#}f^{\#}(a-f_{*}g^{\#}f^{\#}(a))=0$.  Since $f_{*}g^{\#}f^{\#}$ is an isomorphism, $a-f_{*}g^{\#}f^{\#}(a)=0$, which implies $gf=id_{A}$.
\end{proof}

\begin{lemma} \label{lem.noeth}
The schemes ${\mathbb{P}}_{X^{2}}(\mathcal{E})$ and $\Gamma_{d}$ are noetherian.
\end{lemma}

\begin{proof}
The Lemma follows from the fact that $X$ is noetherian and $\mathcal{E}$ is coherent.
\end{proof}
To prove Propositions \ref{prop.rough}, \ref{prop.prelim} and \ref{prop.atv}, which are variations of \cite[Proposition 3.6, 3.7, p. 44-45]{1}, we need a preliminary lemma.

\begin{lemma} \label{lem.easy}
If $(a,b) \in X^{2}$ is a closed point and ${\mathcal{O}}_{(a,b)}$ is the associated ${\mathcal{O}}_{X}$-bimodule, then the bimodule Segre embedding
\begin{equation} \label{eqn.segreuse}
{\mathbb{P}}_{X^{2}}({\mathcal{O}}_{(a,b)}) \otimes_{X} {\mathbb{P}}_{X^{2}}(\mathcal{E}) \rightarrow {\mathbb{P}}_{X^{2}}({\mathcal{O}}_{(a,b)} \otimes_{{\mathcal{O}}_{X}} \mathcal{E})
\end{equation}
is an isomorphism and
$$
\operatorname{length } {\mathcal{O}}_{(a,b)} \otimes_{{\mathcal{O}}_{X}} \mathcal{E}=\operatorname{length }{\mathcal{O}}_{b} \otimes_{{\mathcal{O}}_{X}} \mathcal{E}.
$$
\end{lemma}

\begin{proof}
We first prove (\ref{eqn.segreuse}) is an isomorphism.  Since (\ref{eqn.segreuse}) is a closed immersion of noetherian schemes (Theorem \ref{theorem.segre} and Lemma \ref{lem.noeth}), it suffices to show that there exists an isomorphism between the schemes.

Let $i:a \times X \rightarrow X^{2}$ and $j:b \times X \rightarrow X^{2}$ be the canonical inclusion maps and let $k:X \rightarrow a\times X$ and $l:X \rightarrow b \times X$ be the canonical isomorphisms.  We will show
\begin{equation} \label{eqn.sheavesi}
l^{*}j^{*}\mathcal{E} \cong {\mathcal{O}}_{b} \otimes_{{\mathcal{O}}_{X}} \mathcal{E} \cong k^{*}i^{*}({\mathcal{O}}_{(a,b)} \otimes_{{\mathcal{O}}_{X}} \mathcal{E}).
\end{equation}
This will imply
$$
{\mathbb{P}}_{b \times X}(j^{*}\mathcal{E}) \cong {\mathbb{P}}_{X}(l^{*}j^{*}\mathcal{E}) \cong {\mathbb{P}}_{X}({\mathcal{O}}_{b} \otimes_{{\mathcal{O}}_{X}} \mathcal{E}).
$$
Since
$$
{\mathbb{P}}_{X^{2}}({\mathcal{O}}_{(a,b)}) \otimes_{X} {\mathbb{P}}_{X^{2}}(\mathcal{E}) \cong (b \times X) \times_{X^{2}}{\mathbb{P}}_{X^{2}}(\mathcal{E}) \cong {\mathbb{P}}_{b \times X}(j^{*}\mathcal{E}),
$$
we may conclude that
$$
{\mathbb{P}}_{X^{2}}({\mathcal{O}}_{(a,b)}) \otimes_{X} {\mathbb{P}}_{X^{2}}(\mathcal{E}) \cong {\mathbb{P}}_{X}({\mathcal{O}}_{b} \otimes_{{\mathcal{O}}_{X}} \mathcal{E}).
$$
On the other hand, since the scheme-theoretic support of ${\mathcal{O}}_{(a,b)} \otimes_{{\mathcal{O}}_{X}} \mathcal{E}$ is a closed subscheme of $a \times X$,
$$
{\mathbb{P}}_{X^{2}}({\mathcal{O}}_{(a,b)} \otimes_{{\mathcal{O}}_{X}} \mathcal{E}) \cong {\mathbb{P}}_{a \times X}(i^{*}({\mathcal{O}}_{(a,b)} \otimes_{{\mathcal{O}}_{X}} \mathcal{E})) \cong {\mathbb{P}}_{X}(k^{*}i^{*}({\mathcal{O}}_{(a,b)} \otimes_{{\mathcal{O}}_{X}} \mathcal{E})).
$$
By (\ref{eqn.sheavesi}),
$$
{\mathbb{P}}_{X}(k^{*}i^{*}({\mathcal{O}}_{(a,b)} \otimes_{{\mathcal{O}}_{X}} \mathcal{E})) \cong {\mathbb{P}}_{X}({\mathcal{O}}_{b} \otimes_{{\mathcal{O}}_{X}} \mathcal{E}),
$$
which will establish the first result.

We now prove (\ref{eqn.sheavesi}).  To this end, we will use the fact that if $Z$ is any scheme and $\mathcal{M}$ and $\mathcal{N}$ are quasi-coherent ${\mathcal{O}}_{Z}$-modules with equal and finite support $\{m_{i}\}$, and if ${\mathcal{M}}_{m_{i}} \cong {\mathcal{N}}_{m_{i}}$ as ${\mathcal{O}}_{Z,m_{i}}$-modules for all $i$, then $\mathcal{M} \cong \mathcal{N}$ as ${\mathcal{O}}_{Z}$-modules.

Let ${\mathcal{O}}_{a,a}$ denote the structure sheaf of the closed point $a$ localized at $a$ and let ${\mathcal{O}}_{b,b}$ denote the structure sheaf of the closed point $b$ localized at $b$.  Let $\{c_{i} \}$ be the support of each of the sheaves in (\ref{eqn.sheavesi}).  We note that
$$
(j^{*}\mathcal{E})_{b,c_{i}} \cong ({\mathcal{O}}_{b,b} \otimes {\mathcal{O}}_{X,c_{i}}) \otimes_{{\mathcal{O}}_{X^{2},b,c_{i}}} {\mathcal{E}}_{b,c_{i}}
$$
so that
$$
(l^{*}j^{*}\mathcal{E})_{c_{i}} \cong ({\mathcal{O}}_{b,b} \otimes {\mathcal{O}}_{X,c_{i}}) \otimes_{{\mathcal{O}}_{X^{2},b,c_{i}}} {\mathcal{E}}_{b,c_{i}}
$$
as ${\mathcal{O}}_{X,c_{i}}$-modules.

Next, we note that
\begin{align*}
({\mathcal{O}}_{b} \otimes_{{\mathcal{O}}_{X}} \mathcal{E})_{c_{i}} & = (pr_{2*}(pr_{1}^{*}{\mathcal{O}}_{b} \otimes \mathcal{E}))_{c_{i}} \\
& \cong (pr_{1}^{*}{\mathcal{O}}_{b} \otimes \mathcal{E})_{b,c_{i}} \\
& \cong {\mathcal{O}}_{b,b} \otimes_{{\mathcal{O}}_{X,b}} {\mathcal{O}}_{X,b} \otimes {\mathcal{O}}_{X,c_{i}} \otimes_{{\mathcal{O}}_{X^{2},b,c_{i}}} {\mathcal{E}}_{b,c_{i}} \\
& \cong ({\mathcal{O}}_{b,b}  \otimes {\mathcal{O}}_{X,c_{i}}) \otimes_{{\mathcal{O}}_{X^{2},b,c_{i}}} {\mathcal{E}}_{b,c_{i}}.
\end{align*}
This establishes the left-hand side of (\ref{eqn.sheavesi}).

On the other hand,
$$
(i^{*}({\mathcal{O}}_{(a,b)} \otimes_{{\mathcal{O}}_{X}} \mathcal{E}))_{a,c_{i}} \cong ({\mathcal{O}}_{(a,b)} \otimes_{{\mathcal{O}}_{X}} \mathcal{E})_{a,c_{i}} \otimes_{{\mathcal{O}}_{X^{2},a,c_{i}}} {\mathcal{O}}_{a \times X,a,c_{i}}
$$
and
\begin{align*}
({\mathcal{O}}_{(a,b)} \otimes_{{\mathcal{O}}_{X}} \mathcal{E})_{a,c_{i}} & = (pr_{13*}(pr_{12}^{*}{\mathcal{O}}_{(a,b)} \otimes pr_{23}^{*}\mathcal{E}))_{a,c_{i}} \\
& \cong (pr_{12}^{*}{\mathcal{O}}_{(a,b)} \otimes pr_{23}^{*}\mathcal{E})_{a,b,c_{i}} \\
& \cong {\mathcal{O}}_{a,a} \otimes {\mathcal{O}}_{b,b} \otimes_{{\mathcal{O}}_{X^{2},a,b}} {\mathcal{O}}_{X^{2},a,b} \otimes {\mathcal{O}}_{X,c_{i}} \otimes_{{\mathcal{O}}_{X^{2},b,c_{i}}} {\mathcal{E}}_{b,c_{i}} \\
& \cong {\mathcal{O}}_{a,a} \otimes {\mathcal{O}}_{b,b} \otimes {\mathcal{O}}_{X,c_{i}} \otimes_{{\mathcal{O}}_{X^{2},b,c_{i}}} {\mathcal{E}}_{b,c_{i}}
\end{align*}
Thus,
\begin{align*}
(k^{*}i^{*}({\mathcal{O}}_{(a,b)} \otimes_{{\mathcal{O}}_{X}} \mathcal{E}))_{c_{i}} & \cong ({\mathcal{O}}_{a,a} \otimes {\mathcal{O}}_{b,b} \otimes {\mathcal{O}}_{X,c_{i}} \otimes_{{\mathcal{O}}_{X^{2},b,c_{i}}} {\mathcal{E}}_{b,c_{i}})\otimes_{{\mathcal{O}}_{X^{2},a,c_{i}}} {\mathcal{O}}_{a\times X,a,c_{i}} \\
& \cong {\mathcal{O}}_{b,b} \otimes {\mathcal{O}}_{X,c_{i}} \otimes_{{\mathcal{O}}_{X^{2},b,c_{i}}} {\mathcal{E}}_{b,c_{i}}
\end{align*}
as ${\mathcal{O}}_{X,c_{i}}$-modules.  This establishes (\ref{eqn.sheavesi}).

To show that
$$
\operatorname{length } {\mathcal{O}}_{(a,b)} \otimes_{{\mathcal{O}}_{X}} \mathcal{E}=\operatorname{length }{\mathcal{O}}_{b} \otimes_{{\mathcal{O}}_{X}} \mathcal{E},
$$
we need only note that
$$
\operatorname{length } {\mathcal{O}}_{(a,b)} \otimes_{{\mathcal{O}}_{X}} \mathcal{E}= \operatorname{length } k^{*}i^{*}({\mathcal{O}}_{(a,b)} \otimes_{{\mathcal{O}}_{X}} \mathcal{E}).
$$
The result now follows from (\ref{eqn.sheavesi}).
\end{proof}

\begin{proposition} \label{prop.rough}
If a closed point $p \in \Gamma_{d}$ corresponds to a maximal submodule $\mathcal{M}$ of $\mathcal{E}$ via Proposition \ref{prop.groth}, then the fibres of $p^{d+1}_{1,d}$ and $p^{d+1}_{2,d+1}$ over $p$ are isomorphic to ${\mathbb{P}}_{X^{2}}({\mathcal{E}}^{\otimes d+1}/{\mathcal{I}}_{d+1}+\mathcal{M}\otimes \mathcal{E})$ and ${\mathbb{P}}_{X^{2}}({\mathcal{E}}^{\otimes d+1}/{\mathcal{I}}_{d+1}+\mathcal{E} \otimes \mathcal{M})$ respectively.
\end{proposition}

\begin{proof}
We prove the assertion for $p^{d+1}_{1,d}$.  The proof of the other assertion is similar, so we omit it.

The fibre over $p$, denoted $F(p)$, is the pullback of the diagram
$$
\xymatrix{
& \Gamma_{d+1} \ar[d] \\
{\mathbb{P}}_{X^{2}}({\mathcal{E}}^{\otimes d}/{\mathcal{M}})\otimes_{X}{\mathbb{P}}_{X^{2}}(\mathcal{E}) \ar[r] & {{\mathbb{P}}_{X^{2}}(\mathcal{E})}^{\otimes d+1}.
}
$$
Thus, $F(p)$ is equal to
$$
{\mathbb{P}}_{X^{2}}({\mathcal{E}}^{\otimes d}/{\mathcal{M}})\otimes_{X}{\mathbb{P}}_{X^{2}}(\mathcal{E}) \times_{{\mathbb{P}}_{X^{2}}(\mathcal{E})^{\otimes d+1}}{\mathbb{P}}_{X^{2}}(\mathcal{E})^{\otimes d+1} \times_{{\mathbb{P}}_{X^{2}}({\mathcal{E}}^{\otimes d+1})} {\mathbb{P}}_{X^{2}}({\mathcal{E}}^{\otimes d+1}/{\mathcal{I}}_{d+1})
$$
which is isomorphic to
$$
 {\mathbb{P}}_{X^{2}}({\mathcal{E}}^{\otimes d}/{\mathcal{M}})\otimes_{X}{\mathbb{P}}_{X^{2}}(\mathcal{E})\times_{{\mathbb{P}}_{X^{2}}({\mathcal{E}}^{\otimes d+1})} {\mathbb{P}}_{X^{2}}({\mathcal{E}}^{\otimes d+1}/{\mathcal{I}}_{d+1}).
$$
Since the Segre embedding is functorial (Theorem \ref{theorem.segre}), $F(p)$ is the pullback of
\begin{equation} \label{eqn.newbeee}
\xymatrix{
& {\mathbb{P}}_{X^{2}}({\mathcal{E}}^{\otimes d}/{\mathcal{M}})\otimes_{X}{\mathbb{P}}_{X^{2}}(\mathcal{E}) \ar[d] \\
& {\mathbb{P}}_{X^{2}}(\frac{{\mathcal{E}}^{\otimes d}}{{\mathcal{M}}}\otimes \mathcal{E}) \ar[d] \\
{\mathbb{P}}_{X^{2}}({\mathcal{E}}^{\otimes d+1}/{\mathcal{I}}_{d+1}) \ar[r] & {\mathbb{P}}_{X^{2}}({\mathcal{E}}^{\otimes d+1})
}
\end{equation}
where the top vertical is the Segre embedding and the bottom vertical and horizontal are induced by epimorphisms of ${\mathcal{O}}_{X^{2}}$-modules.  The subdiagram of (\ref{eqn.newbeee}) consisting of its bottom two rows has pullback equal to ${\mathbb{P}}_{X^{2}}({\mathcal{E}}^{\otimes d+1}/{\mathcal{I}}_{d+1}+\mathcal{M}\otimes \mathcal{E})$.  Since the top vertical of (\ref{eqn.newbeee}) is an isomorphism by Lemma \ref{lem.easy},
\begin{equation} \label{eqn.newbeeee}
F(p) \cong {\mathbb{P}}_{X^{2}}({\mathcal{E}}^{\otimes d+1}/{\mathcal{I}}_{d+1}+\mathcal{M}\otimes \mathcal{E}).
\end{equation}
as desired.
\end{proof}

\begin{definition}
If, for every $i \geq 1$, every fibre of $p^{i+1}_{1,i}$ is a finite disjoint union of projectivizations of finite dimensional vector spaces over $k$, then we say $\mathbf{\Gamma}$ {\bf has linear fibres}.
\end{definition}

\begin{lemma} \label{lem.fixer}
Retain the notation in the statement of Proposition \ref{prop.rough}.  If, for all closed $b \in X$, $\operatorname{length }{\mathcal{O}}_{b} \otimes_{{\mathcal{O}}_{X}} \mathcal{E} \leq 2$ and ${\mathcal{I}}_{d+1}$ is not a submodule of $\mathcal{M} \otimes \mathcal{E}$ for $d \geq 1$, then $\Gamma$ has linear fibres.
\end{lemma}

\begin{proof}
Since
$$
\frac{{\mathcal{E}}^{\otimes d+1}}{\mathcal{M}\otimes \mathcal{E} + {\mathcal{I}}_{d+1}} \cong \frac{\frac{{\mathcal{E}}^{\otimes d+1}}{\mathcal{M}\otimes \mathcal{E}}}{\frac{\mathcal{M}\otimes \mathcal{E} + {\mathcal{I}}_{d+1}}{\mathcal{M}\otimes \mathcal{E}}} \cong \frac{\frac{{\mathcal{E}}^{\otimes d}}{\mathcal{M}} \otimes \mathcal{E}}{\mathcal{J}}
$$
for some submodule $\mathcal{J} \subset \frac{{\mathcal{E}}^{\otimes d}}{\mathcal{M}} \otimes \mathcal{E}$,
$$
F(p) \cong {\mathbb{P}}_{X^{2}}\biggl(\frac{\frac{{\mathcal{E}}^{\otimes d}}{\mathcal{M}} \otimes \mathcal{E}}{\mathcal{J}}\biggr).
$$
Since ${\mathcal{I}}_{d+1}$ is not a submodule of $\mathcal{M} \otimes \mathcal{E}$, $\mathcal{J}$ is non-zero.  Since, by Lemma \ref{lem.easy},
$$
\mbox{length }\frac{{\mathcal{E}}^{\otimes d}}{\mathcal{M}} \otimes \mathcal{E}=\mbox{length }{\mathcal{O}}_{b} \otimes \mathcal{E} \leq 2,
$$
it follows that $\operatorname{length }\frac{\frac{{\mathcal{E}}^{\otimes d}}{\mathcal{M}} \otimes \mathcal{E}}{\mathcal{J}} \leq 1$.
\end{proof}
We will use Lemma \ref{lem.fixer} to prove that, for a quantum ruled surface, $\Gamma$ has linear fibres (Proposition \ref{prop.linearfibre}).

\begin{proposition} \label{prop.prelim}
Let $1 \leq m < n \leq d+1$, and let $\pi:\Gamma_{d+1} \rightarrow \Gamma_{n-m+1}$ denote the projection map $p_{m,n}^{d+1}$.  If, for some $p \in \Gamma_{n-m+1}$ the fibre of $\pi$ at $p$, $F(p)$, has the property that either $F(p) \cong p$ or $F(p)$ is empty, then $\pi$ is a closed immersion locally in a neighborhood of $p \in \Gamma_{n-m+1}$.
\end{proposition}

\begin{proof}
Since $\pi$ is projective by Lemma \ref{lem.proper}, it is proper by \cite[Theorem 5.5.3, p.104]{4}.  Since $\Gamma_{l}$ is noetherian by Lemma \ref{lem.noeth}, we may apply Chevalley's Theorem \cite[Proposition 4.4.2, p.136]{5} which tells us that $\pi$ is a finite map in a neighborhood of $p$.  Thus, the localization of $\pi_{*}{\mathcal{O}}_{\Gamma_{d+1}}$ is a finite ${\mathcal{O}}_{\Gamma_{n-m+1},p}$-module.  Since either $F(p) \cong p$ or $F(p)$ is empty, the map ${\mathcal{O}}_{\Gamma_{n-m+1}} \rightarrow {\mathcal{O}}_{\Gamma_{d+1}}$ is surjective locally at $p$ by the Nakayama Lemma.
\end{proof}

\begin{corollary} \label{cor.prelim}
If $p_{1,i}^{i+1}$ is injective on closed points for $d \leq i \leq n$ and $\Gamma$ has linear fibres then $p^{n}_{1,d}$ is a closed immersion.
\end{corollary}

\begin{lemma} \label{lem.rulednew}
Suppose $p^{i+1}_{1,i}$ is a bijection on closed points for $i \geq d$ and $\Gamma$ has linear fibres.  Then there exists an $m \geq d+1$ such that $p^{m'}_{1,m}:\Gamma_{m'} \rightarrow \Gamma_{m}$ is an isomorphism for all $m' \geq m$ and the point-modules over $\mathcal{B}$ are parameterized by a closed subscheme of $\Gamma_{d}$ whose closed points agree with those of $\Gamma_{d}$.
\end{lemma}

\begin{proof}
By Corollary \ref{cor.prelim}, $p^{n}_{1,d}$ is a closed immersion for $n \geq d+1$.  Since $\Gamma_{d}$ is noetherian by Lemma \ref{lem.noeth}, there exists an $m \geq d+1$ such that the image of $\Gamma_{m'}$ in $\Gamma_{d}$ equals the image of $\Gamma_{m'+1}$ in $\Gamma_{d}$ for $m' \geq m$.  Since the closed points of $\Gamma_{i}$ equal the closed points of $\Gamma_{d}$ for $i \geq d$, the assertion follows.
\end{proof}

\begin{proposition} \label{prop.atv}
\begin{enumerate}
\item{}
Assume that, for some $d$, $p_{1,d}^{d+1}$ defines a closed immersion from $\Gamma_{d+1}$ to $\Gamma_{d}$, thus identifying $\Gamma_{d+1}$ with a closed subscheme $E \subset \Gamma_{d}$.  Then $\Gamma_{d+1}$ defines a map $\sigma:E \rightarrow \Gamma_{d}$ such that if $(p_{1}, \ldots, p_{d+1})$ is a point in $E$,
$$
\sigma(p_{1}, \ldots, p_{d})=(p_{2}, \ldots, p_{d+1}),
$$
where $(p_{1}, \ldots, p_{d+1})$ is the unique point of $\Gamma_{d+1}$ lying over $(p_{1}, \ldots, p_{d}) \in \Gamma_{d}$.

\item{}
If in addition $\sigma(E) \subset E$, $\mathcal{I}$ is generated in degree $\leq d$ and $\Gamma$ has linear fibers, then $p_{1,d}^{n}:\Gamma_{n} \rightarrow E$ is an isomorphism for every $n \geq d+1$.
\end{enumerate}
\end{proposition}

\begin{proof}
The first assertion follows from the fact that $E \cong \Gamma_{d+1}$.  To prove the second assertion, we show $p_{1,d}^{n}$ is a bijection of closed points onto $E$.  Our proof closely follows the proof of \cite[Proposition 22.2.10, p.319]{11}.  We first show $p_{1,d}^{n}$ is injective by induction on $n$.  The case $n=d+1$ is true by hypothesis.  We assume the assertion holds for $n$ and prove it for $n+1$.  Let
$$
(p_{1}, \ldots, p_{n+1}), (q_{1}, \ldots, q_{n+1}) \in \Gamma_{n+1}
$$
and suppose $(p_{1}, \ldots, p_{d})=(q_{1}, \ldots, q_{d})$.  Since $(p_{1}, \ldots, p_{n}), (q_{1}, \ldots, q_{n}) \in \Gamma_{n}$ by the induction hypothesis $(p_{1}, \ldots, p_{n})=(q_{1}, \ldots, q_{n})$.  Thus, $(p_{2}, \ldots, p_{d})=(q_{2}, \ldots, q_{d})$.  But these are elements of $E$ since $(p_{2}, \ldots, p_{n+1})$ and $(q_{2}, \ldots, q_{n+1})$ belong to $\Gamma_{n}$.  Thus, $(p_{2}, \ldots, p_{n+1})=(q_{2}, \ldots, q_{n+1})$ by the induction hypothesis.

We now show that $p_{1d}^{n}$ is onto the points of $E$.  Let ${\mathbb{P}} = {\mathbb{P}}_{X^{2}}(\mathcal{E})$.  Since $\mathcal{I}$ is generated in degree $\leq d$, there is a closed immersion over ${\mathbb{P}}^{\otimes n}$
\begin{equation} \label{eqn.induct}
({\mathbb{P}}^{\otimes n-d} \otimes_{X} E) \cap ({\mathbb{P}}^{\otimes n-d-1}\otimes_{X} E \otimes_{X} {\mathbb{P}}) \cap \ldots
\end{equation}
$$
\ldots \cap ({\mathbb{P}} \otimes_{X} E \otimes_{X} {\mathbb{P}}^{\otimes n-d-1}) \cap (E \otimes_{X}{\mathbb{P}}^{\otimes n-d}) \rightarrow \Gamma_{n}
$$
by Corollary \ref{cor.include} applied inductively.  Now, let $p = (p_{1}, \ldots, p_{d}) \in E$.  Define $p_{d+1}, \ldots, p_{n}$ inductively by $(p_{i+1}, \ldots, p_{i+d})=\sigma^{i}(p)$ for $0 \leq i \leq n-d$.  Since $\sigma(E) \subset E$, $(p_{1}, \ldots, p_{n})$ belongs to the left hand side of (\ref{eqn.induct}), hence $(p_{1}, \ldots, p_{n})$ belongs to $\Gamma_{n}$ as desired.  We may conclude that $p_{1,d}^{n}$ is a closed immersion by Corollary \ref{cor.prelim}.

The map $\psi: E \rightarrow {\mathbb{P}}_{X^{2}}(\mathcal{E})^{\otimes n}$ defined by
$$
\psi = \otimes_{i=1}^{n} p_{1,1}^{d+1} (p_{1,d}^{d+1})^{-1} \sigma^{i-1}
$$
induces a closed immersion into each factor of the left hand side of (\ref{eqn.induct}).  Hence, $\psi$ induces a closed immersion to $\Gamma_{n}$ and the result now follows from Lemma \ref{lemma.closed}.
\end{proof}

\section{Quantum Ruled Surfaces}
In order to present Van den Bergh's definition of a quantum ruled surface (Definition \ref{def.qrs}) we need preliminary results regarding affine morphisms and duality.
\subsection{Affine morphisms}
We begin with some notation.  Let $f:Z \rightarrow Y$ be a morphism of noetherian schemes and let $\mathcal{A}$ denote the quasi-coherent sheaf of ${\mathcal{O}}_{Y}$-algebras $f_{*}{\mathcal{O}}_{Z}$.  Define a category ${\Qcoh} \mathcal{A}$ as follows:  the objects of ${\Qcoh} \mathcal{A}$ are the ${\mathcal{O}}_{Y}$-modules with an $\mathcal{A}$-module structure, and the morphisms between two such objects, $\mathcal{M}$ and $\mathcal{N}$ are the elements of $\operatorname{Hom}_{Y}(\mathcal{M},\mathcal{N})$ which are compatible with the $\mathcal{A}$-module structure.

\begin{lemma} \cite[Ex. 5.17e, p.128]{6} \label{lem.equiv}
Suppose $f:Z \rightarrow Y$ is an affine morphism of noetherian schemes.  The functor $f_{\uparrow}: {\Qcoh} Z \rightarrow {\Qcoh} \mathcal{A}$ sending $\mathcal{M}$ to $f_{*}\mathcal{M}$ is an equivalence of categories.
\end{lemma}

\begin{proof}
We define $f^{\uparrow}: {\Qcoh }\mathcal{A} \rightarrow {\Qcoh }Z$ as follows:  if $\{ U_{i} \}$ is an affine open cover of $Y$, $\mathcal{M}$ and $\mathcal{N}$ are $\mathcal{A}$-modules and $\alpha:\mathcal{M} \rightarrow \mathcal{N}$ is an $\mathcal{A}$-module map, let $f^{\uparrow}\mathcal{M}(f^{-1}(U_{i})) = \mathcal{M}(U_{i})$ and let $f^{\uparrow}\alpha(f^{-1}(U))=\alpha(U)$.  Then $f^{\uparrow}\mathcal{M}$ and $f^{\uparrow}\alpha$ are defined by gluing, and it is easy to show that $f^{\uparrow}$ is quasi-inverse to $f_{\uparrow}$.
\end{proof}
Let $F:{\Qcoh }\mathcal{A} \rightarrow {\Qcoh }{\mathcal{O}}_{Y}$ be the forgetful functor.  The following result is easy to prove, so we omit its proof.

\begin{lemma} \label{lem.embed}
With the notation as above, both triangles in the diagram
$$
\xymatrix{
{\Qcoh }{\mathcal{O}}_{Z} \ar@<1ex>[rr]^{f_{\uparrow}} \ar[dr]_{f_{*}} & & {\Qcoh }\mathcal{A} \ar@<1ex>[ll]^{f^{\uparrow}} \ar[dl]^{F} \\
& {\Qcoh }{\mathcal{O}}_{Y}
}
$$
commute up to isomorphism.
\end{lemma}

We next explore the functorial properties of the map sending an affine morphism of noetherian schemes $f:Z \rightarrow Y$ to the functor $f^{\uparrow}:{\Qcoh }\mathcal{A} \rightarrow {\Qcoh }{\mathcal{O}}_{Z}$.  Suppose
$$
\xymatrix{
Z \ar[r]^{f} & Y \ar[r]^{g} & W
}
$$
is a composition of affine morphisms between noetherian schemes.  Let $\mathcal{B}$ denote the sheaf of rings $g_{*}f_{*}{\mathcal{O}}_{Z}$ and let $\mathcal{C}$ denote the sheaf of rings $g_{*}{\mathcal{O}}_{Y}$.  Define categories ${\Qcoh} \mathcal{B}$ and ${\Qcoh} \mathcal{C}$ as before.  Since $f:Z \rightarrow Y$, there is a map of sheaves of rings ${\mathcal{O}}_{Y} \rightarrow f_{*}{\mathcal{O}}_{Z}$, hence a map of sheaves of rings $g_{*}{\mathcal{O}}_{Y}=\mathcal{C} \rightarrow g_{*}f_{*}{\mathcal{O}}_{Z}=\mathcal{B}$.  Thus, any ${\mathcal{B}}$-module has a ${\mathcal{C}}$-module structure via this map.  In addition, if $\alpha$ is an element of $\operatorname{Hom}_{\mathcal{B}}(\mathcal{M},\mathcal{N})$, it has a $\mathcal{B}$-module structure, hence a $\mathcal{C}$-module structure.

\begin{lemma} \label{lem.commutey}
If $I: {\Qcoh }\mathcal{B} \rightarrow {\Qcoh } \mathcal{C}$ is the functor induced by the map of algebras ${\mathcal{B}} \rightarrow {\mathcal{C}}$, then the diagram
$$
\xymatrix{
{\Qcoh }\mathcal{B} \ar[r]^{f_{\uparrow}(gf)^{\uparrow}} \ar[d]_{I} & {\Qcoh }\mathcal{A} \ar[d]^{F} \\
{\Qcoh }\mathcal{C} \ar[r]_{g^{\uparrow}} & {\Qcoh }{\mathcal{O}}_{Y}
}
$$
commutes up to isomorphism.
\end{lemma}

\begin{proof}
Let $\mathcal{M}$ be an object in ${\Qcoh }\mathcal{B}$.  Since $(gf)_{\uparrow}$ and $(gf)^{\uparrow}$ are quasi-inverse, there exists an $\mathcal{N} \in {\Qcoh }Z$ such that $(gf)_{\uparrow}\mathcal{N} \cong \mathcal{M}$.  Now
$$
Ff_{\uparrow}(gf)^{\uparrow}\mathcal{M} \cong f_{*}(gf)^{\uparrow}(gf)_{\uparrow}\mathcal{N} \cong f_{*}\mathcal{N}.
$$
On the other hand, since $I(gf)_{\uparrow} \cong g_{\uparrow}f_{*}$, we have
$$
g^{\uparrow}I(gf)_{\uparrow}\mathcal{N} \cong g^{\uparrow}g_{\uparrow}f_{*}\mathcal{N} \cong f_{*}\mathcal{N}
$$
as desired.
\end{proof}

\begin{corollary} \label{cor.affine}
With the notation as in Lemma \ref{lem.commutey}, $f_{*}(gf)^{\uparrow} \cong g^{\uparrow}I$.
\end{corollary}
If $f:Z \rightarrow Y$ is an affine morphism of noetherian schemes and $\mathcal{M}$ is an ${\mathcal{O}}_{Y}$-module, the ${\mathcal{O}}_{Y}$-module ${\mathcal{H}}{\it om}_{{\mathcal{O}}_{Y}}(f_{*}{\mathcal{O}}_{Z},\mathcal{M})$ has an $f_{*}{\mathcal{O}}_{Z}$-module structure.  In addition, if $\mathcal{N}$ is an ${\mathcal{O}}_{Y}$-module and $\alpha \in \operatorname{Hom}_{Y}(\mathcal{M},\mathcal{N})$, ${\mathcal{H}}{\it om}_{{\mathcal{O}}_{Y}}(f_{*}{\mathcal{O}}_{Z},\alpha)$ is a map of $f_{*}{\mathcal{O}}_{Z}$-modules.  It thus makes sense to define
$$
H: {\Qcoh }Y \rightarrow {\Qcoh }\mathcal{A}
$$
as the functor sending $\mathcal{M}$ to ${\mathcal{H}}{\it om}_{{\mathcal{O}}_{Y}}(f_{*}{\mathcal{O}}_{Z},\mathcal{M})$ with its $f_{*}{\mathcal{O}}_{Z}$-module structure.  We define $f^{!}:{\Qcoh }Y \rightarrow {\Qcoh }Z$ as the composition $f^{\uparrow} H$, and it is not hard to show that $(f_{*},f^{!})$ is an adjoint pair.

\begin{lemma} \label{lem.affine}
Suppose $f:Z \rightarrow Y$ be an affine morphism of noetherian schemes, $\mathcal{F}$ is an ${\mathcal{O}}_{Z}$-module and $\mathcal{G}$ is an ${\mathcal{O}}_{Y}$-module.  The natural isomorphism of ${\mathcal{O}}_{Y}$-modules
\begin{equation} \label{eqn.affine}
f_{*}{\mathcal{H}}{\it om}_{{\mathcal{O}}_{Z}}(\mathcal{F},f^{!}\mathcal{G}) \rightarrow {\mathcal{H}}{\it om}_{{\mathcal{O}}_{Y}}(f_{*}\mathcal{F},\mathcal{G})
\end{equation}
given in \cite[Ex. 6.10b, p.239]{6} is a map of $f_{*}{\mathcal{O}}_{Z}$-modules.
\end{lemma}

\begin{proof}
The map (\ref{eqn.affine}) is given by the composition
$$
f_{*}{\mathcal{H}}{\it om}_{{\mathcal{O}}_{Z}}(\mathcal{F},f^{!}\mathcal{G}) \rightarrow f_{*}{\mathcal{H}}{\it om}_{{\mathcal{O}}_{Z}}(f^{*}f_{*}\mathcal{F},f^{!}\mathcal{G}) \rightarrow
$$
$$
{\mathcal{H}}{\it om}_{{\mathcal{O}}_{Y}}(f_{*}\mathcal{F},f_{*}f^{!}\mathcal{G}) \rightarrow {\mathcal{H}}{\it om}_{{\mathcal{O}}_{Y}}(f_{*}\mathcal{F},\mathcal{G}),
$$
where the first map is induced by the counit $f^{*}f_{*}\mathcal{F} \rightarrow \mathcal{F}$, the last map is induced by the counit $f_{*}f^{!}\mathcal{G} \rightarrow \mathcal{G}$ and the middle map is the obvious natural morphism, which can be shown to be an isomorphism by checking locally.  To show this is actually a homomorphism of $f_{*}{\mathcal{O}}_{Z}$-modules, it suffices to work locally.  If $f:S \rightarrow R$ is a map of rings, $F$ an $R$-module, $G$ an $S$-module, and for $\gamma \in {\Hom}_{S}(R_{S},G)$, we define $\operatorname{ev_{1}}(\gamma)$ (evaluation of $\gamma$ at 1) as $\gamma(1)$, we must show the map
$$
\psi:_{S}{\Hom}_{R}(F,_{R}{\Hom}_{S}(R_{S},G)) \rightarrow {\Hom}_{S}(F_{S},G)
$$
sending $\phi \in {\Hom}_{R}(F,_{R}{\Hom}_{S}(R_{S},G))$ to $\operatorname{ev_{1}} \circ \phi \in _{R}{\Hom}_{S}(F_{S},G)$ is a map of $R$-modules.  Let $r \in R$.  We must check $r \cdot (\operatorname{ev_{1}} \circ \phi) = \operatorname{ev_{1}} (r \cdot \phi)$.  Let $f \in F$.  We have $[r \cdot \phi(f)](1)=\phi(f)(r)$.  On the other hand, $[r \cdot (\operatorname{ev_{1}} \circ \phi)](f)=(\operatorname{ev_{1}} \circ \phi)(rf)=\phi(rf)(1)=r\phi(f)(1)=\phi(f)(r)$ as desired.
\end{proof}

\subsection{The dual of a bimodule}
Let ${\sf bimod }X$ denote the category of coherent bimodules over ${\mathcal{O}}_{X}$ \cite[Definition 2.3, p. 440]{15}.  We present Van den Bergh's definition of the dual of a locally free ${\mathcal{O}}_{X}$-bimodule of finite rank and show that this concept defines a duality, in the following sense, on the full subcategory of ${\sf bimod }X$ consisting of locally free, finite rank ${\mathcal{O}}_{X}$-bimodules.

\begin{definition} \cite[Definition XIV.2.1, p.342]{kas} \label{def.dualityya}
A monoidal category $({\sf C}, \otimes, \mathcal{O})$ with tensor product $\otimes$ and unit $\mathcal{O}$ has a {\bf (left) duality} if for each object $\mathcal{E}$ of $\sf C$ there exists an object ${\mathcal{E}}^{*}$ of $\sf C$ and morphisms $\gamma:{\mathcal{O}} \rightarrow \mathcal{E} \otimes {\mathcal{E}}^{*}$ and $\delta:{\mathcal{E}}^{*} \otimes \mathcal{E} \rightarrow {\mathcal{O}}$ in $\sf C$ such that
\begin{equation} \label{eqn.firsteqn}
(\operatorname{id}_{\mathcal{E}} \otimes \delta)(\gamma \otimes \operatorname{id}_{\mathcal{E}})=\operatorname{id}_{\mathcal{E}}
\end{equation}
and
\begin{equation} \label{eqn.secondeqn}
(\delta \otimes \operatorname{id}_{{\mathcal{E}}^{*}})(\operatorname{id}_{{\mathcal{E}}^{*}}\otimes \gamma)=\operatorname{id}_{{\mathcal{E}}^{*}}.
\end{equation}
\end{definition}
We then show that duality extends to an endofunctor on ${\sf bimod} X$.

\subsubsection{The dual of a locally free bimodule}
\begin{lemma} \label{lemma.cm}
Let $f:(R, \mathfrak{m}) \rightarrow (S, \mathfrak{n})$ be a finite local map of noetherian commutative local rings.  A finitely generated $S$-module of maximal dimension is Cohen-Macaulay as an $S$-module if and only if it is Cohen-Macaulay as an $R$-module.
\end{lemma}

\begin{proof}
Let $M$ be a finitely generated $S$ module of maximal dimension.  We first show that $\operatorname{dim} M_{R}=\operatorname{dim} M$.  Since $M$ has maximal dimension, it suffices to show that $\operatorname{dim} M_{R} \geq \operatorname{dim} M$.  To prove this fact, suppose ${\mathfrak{p}}_{1} \not\subseteq {\mathfrak{p}}_{2}$ are primes in $S$.  Since $f$ is finite, $S$ is integral over $R$.  Hence $R \cap {\mathfrak{p}}_{1} \neq R \cap {\mathfrak{p}}_{2}$.

Since $\operatorname{depth}M=\operatorname{depth}M_{R}$ \cite[Ex. 1.2.26b, p.15]{3} the assertion follows.
\end{proof}
Suppose $X$ and $Y$ are $k$-schemes, and let $pr_{i}:X \times Y \rightarrow X,Y$ denote the standard projections.

\begin{corollary} \label{cor.freee} \cite[Proposition 4.1.6]{14}
Let $X$ and $Y$ be smooth schemes of the same dimension and suppose $\mathcal{E}$ is a coherent ${\mathcal{O}}_{X}-{\mathcal{O}}_{Y}$-bimodule which is locally free of finite rank on one side.  Then $\mathcal{E}$ is locally free of finite rank on the other side as well.
\end{corollary}

\begin{proof}
Suppose $\mathcal{E}$ is locally free on the left.  Set $Z = \operatorname{Supp} \mathcal{E}$.  The pullback of $\mathcal{E}$ to $Z$ has dimension equal to that of $Z$.  Let $\alpha$ and $\beta$ denote the restrictions of $pr_{i}$ to $Z$.  Since $\alpha$ is finite and $pr_{1*}\mathcal{E}$ is Cohen-Macaulay, the pullback of $\mathcal{E}$ to $Z$ is Cohen-Macaulay by Lemma \ref{lemma.cm}.  Since $\beta$ is finite, then again by Lemma \ref{lemma.cm}, $pr_{2*}\mathcal{E}$ is Cohen-Macaulay.  Since $Y$ is smooth, $pr_{2*}\mathcal{E}$ is locally free of finite rank.
\end{proof}
We recall some notation from \cite{15}.  If $X$, $Y$ and $Z$ are schemes, $\alpha:Z \rightarrow X$ and $\beta:Z \rightarrow Y$ are morphisms and $\mathcal{H}$ is an ${\mathcal{O}}_{Z}$-module, we define $_{\alpha}\mathcal{H}_{\beta} = (\alpha,\beta)_{*}\mathcal{H}$.
\begin{definition} \label{def.duality}
(Duality).  Let $X$ and $Y$ be noetherian schemes and suppose $\mathcal{E}$ is an ${\mathcal{O}}_{X}-{\mathcal{O}}_{Y}$-bimodule.  If $Z=\operatorname{Supp }\mathcal{E}$ and $\mathcal{H}$ is the pullback of $\mathcal{E}$ to $Z$, and if $\alpha$ and $\beta$ are the restrictions of the projections $pr_{i}:X \times Y \rightarrow X,Y$ to $Z$, we define the {\bf dual} of $\mathcal{E}$ to be the ${\mathcal{O}}_{Y}-{\mathcal{O}}_{X}$-bimodule
$$
{\mathcal{E}}^{*} = _{\beta}[\beta^{\uparrow}{\mathcal{H}}{\it om}_{{\mathcal{O}}_{Y}}(\beta_{*}\mathcal{H},{\mathcal{O}}_{Y})]_{\alpha}.
$$
\end{definition}

\begin{lemma} \cite[Corollary 4.1.9]{14}
Suppose $X$ and $Y$ are smooth schemes of the same dimension.  If $\mathcal{E}$ is a coherent ${\mathcal{O}}_{X}-{\mathcal{O}}_{Y}$-bimodule which is locally free of rank $n$ on both sides then ${\mathcal{E}}^{*}$ is locally free of rank $n$ on both sides.
\end{lemma}

\begin{lemma} \label{lem.projform}
\cite[Proposition 2.2(5), p. 440]{15} (Projection Formula) Let $f:U \rightarrow W$ be a map of noetherian schemes.  If $\mathcal{M}$ is an ${\mathcal{O}}_{U}$-module, relatively locally finite for $f$, and $\mathcal{N}$ is a quasi-coherent ${\mathcal{O}}_{W}$-module then the natural map
\begin{equation} \label{eqn.proj}
f_{*}\mathcal{M} \otimes_{{\mathcal{O}}_{W}} \mathcal{N} \rightarrow f_{*}(\mathcal{M} \otimes_{{\mathcal{O}}_{Y}}f^{*}\mathcal{N})
\end{equation}
is an isomorphism.
\end{lemma}

\begin{proof}
The map (\ref{eqn.proj}) is the composition of natural maps
$$
f_{*}\mathcal{M} \otimes_{{\mathcal{O}}_{W}} \mathcal{N} \rightarrow f_{*}f^{*}(f_{*}\mathcal{M} \otimes_{{\mathcal{O}}_{W}} \mathcal{N}) \rightarrow
$$
$$
f_{*}(f^{*}f_{*}\mathcal{M} \otimes_{{\mathcal{O}}_{W}} f^{*}\mathcal{N}) \rightarrow f_{*}(\mathcal{M} \otimes_{{\mathcal{O}}_{W}} f^{*}\mathcal{N})
$$
each of which is induced by the unit or counit of the adjoint pair $(f^{*},f_{*})$.
\end{proof}

\begin{lemma} \label{lem.projformextend}
Let $f:U \rightarrow W$ be an affine map of schemes.  If $\mathcal{M}$ is an ${\mathcal{O}}_{U}$-module and $\mathcal{N}$ is an ${\mathcal{O}}_{W}$-module then the map (\ref{eqn.proj}) is a map of $f_{*}{\mathcal{O}}_{U}$-modules.
\end{lemma}

\begin{proof}
We describe (\ref{eqn.proj}) locally.  Let $f:S \rightarrow R$ be a map of rings, let $M$ be an $S$-module and let $N$ be an $R$-module.  Let $m \in M$ and $n \in N$.  We claim the map of $R$-modules
\begin{equation} \label{eqn.proja}
\phi: M_{R} \otimes_{R} N \rightarrow _{R}((S \otimes_{R}N )\otimes_{S}M)
\end{equation}
sending $(m \otimes n)$ to $((1 \otimes n)\otimes m)$ is an $S$-module map.  Since (\ref{eqn.proja}) is just the affine version of (\ref{eqn.proj}), this would complete the demonstration.  If $s \in S$,
$$
s \cdot \phi(m \otimes n)=s\cdot ((1 \otimes n) \otimes m)=(1 \otimes n) \otimes sm.
$$
On the other hand,
$$
\phi(s \cdot (m \otimes n))=\phi(sm \otimes n)=(1 \otimes n) \otimes sm.
$$
Thus, (\ref{eqn.proj}) is a map of $f_{*}{\mathcal{O}}_{U}$-modules.
\end{proof}

\begin{lemma} \label{lem.locfree}
Let $f:U \rightarrow W$ be an affine map of noetherian schemes.  If $\mathcal{M}$ is an ${\mathcal{O}}_{U}$-module such that $f_{*}\mathcal{M}$ is locally free of finite rank, and $\mathcal{N}$ is a quasi-coherent ${\mathcal{O}}_{W}$-module, the natural isomorphism
\begin{equation} \label{eqn.nfiso}
\mathcal{N} \otimes_{{\mathcal{O}}_{W}}{\mathcal{H}}{\it om}_{{\mathcal{O}}_{W}}(f_{*}\mathcal{M},{\mathcal{O}}_{W}) \rightarrow {\mathcal{H}}{\it om}(f_{*}\mathcal{M},\mathcal{N})
\end{equation}
\cite[Ex. 5.1b, p.123]{6} is an isomorphism of $f_{*}{\mathcal{O}}_{U}$-modules.
\end{lemma}

\begin{proof}
We describe (\ref{eqn.nfiso}) locally.  Let $S \rightarrow R$ be a map of rings, let $M$ be an $R$-module such that $M_{S}$ is a free $S$-module of finite rank $l$ with isomorphism $\phi:M_{S} \rightarrow S^{\oplus l}$, and let $N$ be an $S$-module.  Finally, let $n \in N$ and let $\psi \in \operatorname{Hom}_{S}(M_{S},S)$.  Then (\ref{eqn.nfiso})
$$
\delta: N \otimes_{S} \operatorname{Hom}_{S}(M_{S},S) \rightarrow \operatorname{Hom}_{S}(M_{S},N)
$$
sends $n \otimes \psi$ to $\gamma \in \operatorname{Hom}_{S}(M_{S},N)$ such that $\gamma (\phi^{-1}(1, \ldots, 1))= \psi(\phi^{-1}(1, \ldots, 1))n$.  Thus, if $r \in R$,
$$
\delta(r \cdot n \otimes \psi)=\delta(n \otimes r \cdot \psi),
$$
and
$$
\delta(n \otimes r \cdot \psi)(\phi^{-1}(1, \ldots, 1))=\psi(r\phi^{-1}(1, \ldots, 1))n.
$$
On the other hand
$$
r \cdot \delta(n \otimes \psi)(\phi^{-1}(1, \ldots, 1))=\delta(n \otimes \psi)(r \phi^{-1}(1, \ldots, 1))=\psi(r \phi^{-1}(1, \ldots, 1))n
$$
which is just what we needed to show.
\end{proof}
The proof of the following Proposition is due to Van den Bergh.
\begin{proposition} \label{prop.adjoint}
Let $X$ be a noetherian scheme and let $\mathcal{E}$ be a coherent, locally free ${\mathcal{O}}_{X}$-bimodule of finite rank.  The functors
$$
- \otimes_{{\mathcal{O}}_{X}} \mathcal{E}: {\Qcoh }X \rightarrow {\Qcoh }Y
$$
and
$$
- \otimes_{{\mathcal{O}}_{Y}}{\mathcal{E}}^{*}:{\Qcoh }Y \rightarrow {\Qcoh }X
$$
form an adjoint pair $(- \otimes_{{\mathcal{O}}_{X}} \mathcal{E}, - \otimes_{{\mathcal{O}}_{Y}}{\mathcal{E}}^{*})$.
\end{proposition}

\begin{proof}
Let $\mathcal{A}$ be an ${\mathcal{O}}_{X}$-module and let $\mathcal{B}$ be an ${\mathcal{O}}_{Y}$-module.  We show that
\begin{equation} \label{eqn.adjoint}
\operatorname{Hom}_{Y}({\mathcal{A}}\otimes_{{\mathcal{O}}_{X}}\mathcal{E},\mathcal{B}) \cong \operatorname{Hom}_{X}({\mathcal{A}}, {\mathcal{B}}\otimes_{{\mathcal{O}}_{Y}}{\mathcal{E}}^{*}).
\end{equation}
The left hand side of (\ref{eqn.adjoint}) is
\begin{align*}
\mbox{Hom}_{Y}(\beta_{*}(\alpha^{*}{\mathcal{A}} \otimes_{{\mathcal{O}}_{Z}} {\mathcal{H}}),\mathcal{B}) & \cong \mbox{Hom}_{Z}(\alpha^{*}\mathcal{A} \otimes_{{\mathcal{O}}_{Z}} \mathcal{H},\beta^{!}\mathcal{B})  \\
& \cong \mbox{Hom}_{Z}(\alpha^{*}\mathcal{A}, \mathcal{H}\it{om}_{{\mathcal{O}}_{Z}}(\mathcal{H},\beta^{!}\mathcal{B})) \\
& \cong \mbox{Hom}_{Z}(\mathcal{A}, \alpha_{*}\mathcal{H}\it{om}_{{\mathcal{O}}_{Z}}(\mathcal{H},\beta^{!}\mathcal{B})).
\end{align*}
On the other hand, the right hand side of (\ref{eqn.adjoint}) is
$$
\mbox{Hom}_{X}(\mathcal{A}, \alpha_{*}(\beta^{*}\mathcal{B} \otimes_{{\mathcal{O}}_{Z}} \beta^{\uparrow}{\mathcal{H}}\it{om}_{{\mathcal{O}}_{Y}}(\beta_{*}\mathcal{H},{\mathcal{O}}_{Y}))).
$$
Thus, to prove that (\ref{eqn.adjoint}) holds, we must show that
\begin{equation} \label{eqn.equiv}
{\mathcal{H}}\it{om}_{{\mathcal{O}}_{Z}}(\mathcal{H},\beta^{!}\mathcal{B}) \cong \beta^{*}\mathcal{B} \otimes_{{\mathcal{O}}_{Z}} \beta^{\uparrow}{\mathcal{H}}\it{om}_{{\mathcal{O}}_{Y}}(\beta_{*}\mathcal{H},{\mathcal{O}}_{Y}).
\end{equation}
To prove this fact, we note that
\begin{align*}
\beta_{*}(\beta^{*}\mathcal{B} \otimes_{{\mathcal{O}}_{Z}} \beta^{\uparrow}{\mathcal{H}}\it{om}_{{\mathcal{O}}_{Y}}(\beta_{*}\mathcal{H},{\mathcal{O}}_{Y})) & \cong \mathcal{B} \otimes_{{\mathcal{O}}_{Y}} \beta_{*}\beta^{\uparrow}{\mathcal{H}}\it{om}_{{\mathcal{O}}_{Y}}(\beta_{*}\mathcal{H},{\mathcal{O}}_{Y}) \\
& \cong \mathcal{B} \otimes_{{\mathcal{O}}_{Y}} {\mathcal{H}}\it{om}_{{\mathcal{O}}_{Y}}(\beta_{*}\mathcal{H},{\mathcal{O}}_{Y}) \\
& \cong {\mathcal{H}}\it{om}_{{\mathcal{O}}_{Y}}(\beta_{*}\mathcal{H},{\mathcal{O}}_{Y}) \\
& \cong \beta_{*}{\mathcal{H}}{\it om}_{{\mathcal{O}}_{Z}}(\mathcal{H},\beta^{!}\mathcal{B})
\end{align*}
where the first isomorphism is the projection formula (\ref{eqn.proj}), the second isomorphism is induced by the natural isomorphism $\beta_{*}\beta^{!} \cong I$ (\ref{lem.embed}), the third isomorphism is (\ref{eqn.nfiso}) and the final isomorphism is (\ref{eqn.affine}).  By Lemmas \ref{lem.projform}, \ref{lem.locfree}, and \ref{lem.affine}, the composition of these four isomorphisms is an isomorphism of $\beta_{*}{\mathcal{O}}_{Z}$-modules.  Thus, by Lemma \ref{lem.equiv} there is an isomorphism (\ref{eqn.equiv}).
\end{proof}
Compare with \cite[p.6]{14}
\begin{corollary} \label{cor.dualityprime}
The map $(-)^{*}$ is a duality on the full subcategory of ${\sf bimod} X$ consisting of locally free, finite rank ${\mathcal{O}}_{X}$-bimodules is a duality in the sense of Definition \ref{def.dualityya}.
\end{corollary}

\begin{proof}
If $\mathcal{E}$ and $\mathcal{F}$ are locally free, finite rank ${\mathcal{O}}_{X}$-bimodules, then $\mathcal{E} \otimes_{{\mathcal{O}}_{X}} \mathcal{F}$ is also a locally free, finite rank ${\mathcal{O}}_{X}$-bimodule \cite[Lemma 4.1.7]{14}, so that the full subcategory of ${\sf bimod} X$ consisting of locally free, finite rank ${\mathcal{O}}_{X}$-bimodules is a monoidal category.

Let $\mathcal{E}$ be a locally free, finite rank ${\mathcal{O}}_{X}$-bimodule and let $(\eta, \epsilon)$ be the unit and counit of the adjoint pair $(- \otimes \mathcal{E},- \otimes {\mathcal{E}}^{*})$ (Proposition \ref{prop.adjoint}).  If we let $\gamma$ denote the composition
$$
\xymatrix{
{\mathcal{O}}_{X} \ar[rr]^{\eta_{{\mathcal{O}}_{X}}} & & {\mathcal{O}}_{X} \otimes \mathcal{E} \otimes {\mathcal{E}}^{*} \ar[rr]^{\cong} & & \mathcal{E} \otimes {\mathcal{E}}^{*}
}
$$
and we let $\delta$ denote the composition
$$
\xymatrix{
{\mathcal{E}}^{*} \otimes \mathcal{E} \ar[rr]^{\cong} & & {\mathcal{O}}_{X} \otimes {\mathcal{E}}^{*} \otimes \mathcal{E} \ar[rr]^{\epsilon_{{\mathcal{O}}_{X}}} & & {\mathcal{O}}_{X}
}
$$
an easy computation shows that $\gamma$ and $\delta$ satisfy (\ref{eqn.firsteqn}) and (\ref{eqn.secondeqn}).
\end{proof}

\subsubsection{Duality for bimodules}
We now show that duality (\ref{def.duality}) defines a functor $(-)^{*}:{\sf{bimod}}X \rightarrow {\sf{bimod }}X$.

\begin{lemma} \label{lem.same}
Suppose
$$
\xymatrix{
& W \ar[d]^{\delta} & \\
& Z \ar[dl]_{\alpha} \ar[dr]^{\beta} & \\
X & & X
}
$$
is a diagram of finite maps between schemes.  If $\mathcal{K}$ is an object of ${\Mod }W$ then
$$
_{\beta}[\beta^{\uparrow}{\mathcal{H}}{\it om}_{{\mathcal{O}}_{X}}(\beta_{*}\delta_{*}\mathcal{K}, {\mathcal{O}}_{X})]_{\alpha} \cong _{\beta \delta}[(\beta \delta)^{\uparrow}{\mathcal{H}}{\it om}_{{\mathcal{O}}_{X}}((\beta \delta)_{*}\mathcal{K},{{\mathcal{O}}_{X}})]_{\alpha \delta}.
$$
\end{lemma}

\begin{proof}
By \cite[Lemma 2.8(1), p. 443]{15},
$$
_{\beta \delta}[(\beta \delta)^{\uparrow}{\mathcal{H}}{\it om}_{{\mathcal{O}}_{X}}((\beta \delta)_{*}\mathcal{K},{{\mathcal{O}}_{X}})]_{\alpha \delta} \cong _{\beta}[\delta_{*}(\beta \delta)^{\uparrow}{\mathcal{H}}{\it om}_{{\mathcal{O}}_{X}}((\beta \delta)_{*} \mathcal{K},{{\mathcal{O}}_{X}})]_{\alpha}.
$$
Thus, to prove the lemma, we must show
\begin{equation} \label{eqn.complicated}
\delta_{*}(\beta \delta)^{\uparrow}{\mathcal{H}}{\it om}_{{\mathcal{O}}_{X}}(\beta_{*}\delta_{*}\mathcal{K},{\mathcal{O}}_{X}) \cong \beta^{\uparrow}{\mathcal{H}}{\it om}_{{\mathcal{O}}_{X}}(\beta_{*}\delta_{*}\mathcal{K}, {\mathcal{O}}_{X}).
\end{equation}
Let $\mathcal{A}$ be the sheaf of rings $\beta_{*}\delta_{*}{{\mathcal{O}}_{W}}$ and let $\mathcal{B}$ be the sheaf of rings $\beta_{*}{{\mathcal{O}}_{Z}}$.  Let ${\Qcoh }\mathcal{A}$ and ${\Qcoh }\mathcal{B}$ be the categories defined at the beginning of this section.  Let
$$
J: {\Qcoh }\mathcal{A} \rightarrow {\Qcoh }\mathcal{B}
$$
be induced by the map of sheaves of rings $\mathcal{B} \rightarrow \mathcal{A}$.  Since $\delta_{*}(\beta \delta)^{\uparrow} \cong \beta^{\uparrow}J$ by Corollary \ref{cor.affine}, and since $J {\mathcal{H}}{\it om}_{{\mathcal{O}}_{X}}(\beta_{*}\delta_{*}\mathcal{K}, {\mathcal{O}}_{X})$ is the module ${\mathcal{H}}{\it om}_{{\mathcal{O}}_{X}}(\beta_{*}\delta_{*}\mathcal{K}, {\mathcal{O}}_{X})$ with its $\beta_{*}{\mathcal{O}}_{Z}$-module structure, (\ref{eqn.complicated}) holds as desired.
\end{proof}

\begin{corollary}
$(-)^{*}:{\sf bimod }X \rightarrow {\sf bimod }X$ is a functor.
\end{corollary}

\begin{proof}
Suppose $\psi: \mathcal{E} \rightarrow \mathcal{F}$ is a map of coherent bimodules with supports $W$ and $Z$ respectively.  Let $(\alpha, \beta):W \cup Z \rightarrow X \times Y$ be the inclusion map.  We define $\psi^{*}:{\mathcal{E}}^{*} \rightarrow {\mathcal{F}}^{*}$ as the composition
$$
\xymatrix{
{\mathcal{E}}^{*} \ar[r] & _{\beta}[\beta^{\uparrow}{\mathcal{H}}{\it om}_{{\mathcal{O}}_{X}}(\beta_{*}\beta^{*}\mathcal{E},{\mathcal{O}}_{X})]_{\alpha} \ar[d]^{_{\beta}[\beta^{\uparrow}{\mathcal{H}}{\it om}_{{\mathcal{O}}_{X}}(\beta_{*}\beta^{*}\psi,{\mathcal{O}}_{X})]_{\alpha}} \\
&  _{\beta}[\beta^{\uparrow}{\mathcal{H}}{\it om}_{{\mathcal{O}}_{X}}(\beta_{*}\beta^{*}\mathcal{F},{\mathcal{O}}_{X})]_{\alpha} \ar[r]  &  {\mathcal{F}}^{*}
}
$$
whose first and last arrows are the isomorphisms from Lemma \ref{lem.same}.
\end{proof}

\subsection{The definition of a quantum ruled surface}

\begin{definition} \label{def.nondeg}
Let $X$ be a smooth curve over $k$, and let $\mathcal{E}$ be a locally free, rank two ${\mathcal{O}}_{X}$-bimodule.  An invertible bimodule $\mathcal{Q} \subset \mathcal{E} \otimes_{{\mathcal{O}}_{X}} \mathcal{E}$ is {\bf nondegenerate} if the composition
\begin{equation} \label{eqn.nondegen}
{\mathcal{E}}^{*}\otimes_{{\mathcal{O}}_{X}}\mathcal{Q} \rightarrow {\mathcal{E}}^{*} \otimes_{{\mathcal{O}}_{X}} {\mathcal{E}}\otimes_{{\mathcal{O}}_{X}} {\mathcal{E}} \rightarrow {\mathcal{E}}
\end{equation}
whose first composite is induced by inclusion and whose second composite is induced by the counit map of the pair $(- \otimes_{{\mathcal{O}}_{X}} {\mathcal{E}}, - \otimes_{{\mathcal{O}}_{X}} {\mathcal{E}}^{*})$, is an isomorphism.
\end{definition}
Suppose $\mathcal{B}$ is a bimodule algebra \cite[Section 4, p. 453]{15}.  Let ${\sf{Grmod }}\mathcal{B}$ denote the category of ${\mathbb{N}}$-graded right $\mathcal{B}$-modules, let ${\sf tors}$ denote the full subcategory of ${\sf Grmod }\mathcal{B}$ consisting of coherent graded $\mathcal{B}$-modules $\mathcal{M}$ such that ${\mathcal{M}}_{n}=0$ for $n >> 0$ and let ${\sf Tors}$ denote the closure of ${\sf tors}$ under direct limits.  Finally, let ${\sf Proj} \mathcal{B}$ denote the quotient category ${\sf Grmod }\mathcal{B}/{\sf Tors}$.

\begin{definition} \label{def.qrs}
Let $X$ be a smooth curve over an algebraically closed field $k$.  A {\bf quantum ruled surface over $X$} is a category ${\sf Proj} T(\mathcal{E})/(\mathcal{Q})$, where $\mathcal{E}$ is a locally free, rank two ${\mathcal{O}}_{X}$-bimodule, and $\mathcal{Q} \subset \mathcal{E}\otimes_{{\mathcal{O}}_{X}} \mathcal{E}$ is nondegenerate.
\end{definition}

\begin{theorem} \label{theorem.homological} \cite[Theorem 7.1.2]{14}
Suppose $\mathcal{B} = T(\mathcal{E})/(\mathcal{Q})$ and ${\sf Proj} \mathcal{B}$ is a quantum ruled surface.  Then ${\mathcal{B}}_{i}$ is locally free of rank $i+1$ for $i \geq 0$ and multiplication induces an ${\mathcal{O}}_{X}-\mathcal{B}$-module resolution
$$
0 \rightarrow \mathcal{Q} \otimes \mathcal{B}(-2) \rightarrow \mathcal{E} \otimes \mathcal{B}(-1) \rightarrow \mathcal{B} \rightarrow {\mathcal{O}}_{\Delta} \rightarrow 0
$$
where $\Delta \subset X \times X$ is the diagonal and ${\mathcal{O}}_{\Delta}$ is concentrated in degree zero.
\end{theorem}

\section{Point Modules over Quantum Ruled Surfaces}
Suppose $\operatorname{Proj} \mathcal{B}$ is a quantum ruled surface.  We parameterize point modules over $\mathcal{B}$.  While our parameterization follows from work of Van den Bergh, our proof, which is distinct from Van den Bergh's proof, serves to illustrate the utility of the general results from Section 2.  We also construct ${\mathcal{O}}_{X}-\mathcal{B}$-module resolutions for point modules when $X={\mathbb{P}}^{1}$.
\subsection{Parameterizations of point modules over quantum ruled surfaces}
\begin{lemma} \label{lem.ktheory}
If $p \in X$ is a closed point, $\operatorname{length} {\mathcal{O}}_{p} \otimes {\mathcal{B}}_{i} = i+1$.
\end{lemma}

\begin{proof}
By definition, we have ${\mathcal{O}}_{p} \otimes {\mathcal{B}}_{i} = pr_{2*}(pr_{1}^{*}({\mathcal{O}}_{p}) \otimes {\mathcal{B}}_{i})$.  Since the support of ${\mathcal{B}}_{i}$ is finite on both sides and $k$ is algebraically closed, we have
\begin{align*}
\operatorname{length } pr_{2*}(pr_{1}^{*}({\mathcal{O}}_{p}) \otimes {\mathcal{B}}_{i}) = & \operatorname{length } pr_{1}^{*}({\mathcal{O}}_{p}) \otimes {\mathcal{B}}_{i} \\
= & \operatorname{length } pr_{1*}(pr_{1}^{*}({\mathcal{O}}_{p}) \otimes {\mathcal{B}}_{i}) \\
= & \operatorname{length }{\mathcal{O}}_{p} \otimes pr_{1*} {\mathcal{B}}_{i} \\
= & i+1.
\end{align*}
where the last equality is a consequence of Theorem \ref{theorem.homological}.
\end{proof}
The following Proposition is due to S. Paul Smith.

\begin{proposition} \label{prop.smith}
Suppose $d \geq 1$ is an integer.  Given a ${\mathcal{B}}$-point module $\mathcal{M}=\oplus_{i=0}^{d-1} {\mathcal{M}}_{i}$ of length $d$, there exists a unique isomorphism class of $\mathcal{B}$-point module $[\mathcal{N}]$ of length $d+1$ such that the projection of any module in the class $[\mathcal{N}]$ to its first $d$ components is isomorphic to $\mathcal{M}$.
\end{proposition}

\begin{proof}
Let $E$ denote the functor $- \otimes_{{\mathcal{O}}_{X}} \mathcal{E}$ and let $E^{*}$ denote the functor $- \otimes_{{\mathcal{O}}_{X}} {\mathcal{E}}^{*}$.  Since $\mathcal{M}$ is a point module, there is an epi $\mu_{d-2}:E({\mathcal{M}}_{d-2}) \rightarrow {\mathcal{M}}_{d-1}$.  Applying the functor $E^{*}$ to $\mu_{d-2}$ we get a map
$$
E^{*}(\mu_{d-2}):E^{*}E({\mathcal{M}}_{d-2}) \rightarrow E^{*}({\mathcal{M}}_{d-1}).
$$
Since $(E,E^{*})$ is an adjoint pair, there is a unit map $\eta:{\mathcal{M}}_{d-2} \rightarrow E^{*}E({\mathcal{M}}_{d-2})$.

We claim the composition
\begin{equation} \label{eqn.immonic}
E^{*}(\mu_{d-2}) \eta:{\mathcal{M}}_{d-2} \rightarrow E^{*}({\mathcal{M}}_{d-1})
\end{equation}
is monic.  For, since ${\mathcal{M}}_{d-2}$ is simple, to show $E^{*}(\mu_{d-2}) \eta $ is monic it suffices to show that it is nonzero.  By adjointness of $(E,E^{*})$, there exists an isomorphism
$$
\nu:\operatorname{Hom}_{{\mathcal{O}}_{X}}({\mathcal{M}}_{d-2},E^{*}({\mathcal{M}}_{d-1})) \rightarrow \operatorname{Hom}_{{\mathcal{O}}_{X}}(E({\mathcal{M}}_{d-2}),{\mathcal{M}}_{d-1}).
$$
Thus, to show that $E^{*}(\mu_{d-2}) \eta$ is monic, it suffices to show that $\nu(E^{*}(\mu_{d-2})) \eta $ is nonzero.  But $\nu(E^{*}(\mu_{d-2})) \eta = \mu_{d-2}$ which is nonzero since it is epi.  Thus, $E^{*}(\mu_{d-2}) \eta $ is monic.

By Lemma \ref{lem.ktheory}, $E^{*}{\mathcal{M}}_{d-1}$ has length two over $k$, so the cokernel of $E^{*}(\mu_{d-2}) \eta$, $\mathcal{N}$, is simple.  Thus, we have an epi $E^{*}{\mathcal{M}}_{d-1} \rightarrow \mathcal{N}$.  Let $Q$ denote the functor $-\otimes_{{\mathcal{O}}_{X}}\mathcal{Q}$.  Since $\mathcal{Q}$ is nondegenerate, $QE^{*} \cong E$ giving an epi $\mu_{d-1}:E({\mathcal{M}}_{d-1}) \rightarrow Q(\mathcal{N})$.  Let ${\mathcal{M}}_{d} = Q(\mathcal{N})$.  We claim $\oplus_{i=0}^{d}{\mathcal{M}}_{i}$, with $d$th multiplication $\mu_{d-1}$, is a truncated ${\mathcal{B}}$-point module of length $d+1$.  We need only note that the composition
$$
\xymatrix{
Q({\mathcal{M}}_{d-2}) \ar[r]^{Q\eta} &  QE^{*}E({\mathcal{M}}_{d-2}) \ar[rr]^{QE^{*}\mu_{d-2}} & & QE^{*}({\mathcal{M}}_{d-1})  \ar[r] & {\mathcal{M}}_{d}\\
& EE({\mathcal{M}}_{d-2}) \ar[rr]_{E\mu_{d-2}} \ar[u] & & E({\mathcal{M}}_{d-1}) \ar[u]
}
$$
is zero since the top row is zero and the center square commutes by naturality of the isomorphism $QE^{*} \cong E$.  The result follows.
\end{proof}

\begin{proposition} \label{prop.linearfibre}
$\Gamma$ has linear fibres.
\end{proposition}

\begin{proof}
By Lemmas \ref{lem.fixer} and \ref{lem.ktheory}, it suffices to show that $(Q)_{d+1}$ is not contained in $\mathcal{M}\otimes \mathcal{E}$ for any maximal submodule $\mathcal{M} \subset {\mathcal{E}}^{\otimes d}$.

Suppose $(Q)_{d+1} \subset \mathcal{M}\otimes \mathcal{E}$.  Then the composition
$$
{\mathcal{E}}^{\otimes d-1} \otimes \mathcal{Q} \rightarrow {\mathcal{E}}^{\otimes d-1} \otimes \mathcal{E} \otimes \mathcal{E} \rightarrow \frac{{\mathcal{E}}^{\otimes d}}{\mathcal{M}} \otimes \mathcal{E}
$$
equals $0$.  Tensoring this sequence on the right by ${\mathcal{Q}}^{-1}$ yields the top row of a commutative diagram
\begin{equation} \label{eqn.nonnon}
\xymatrix{
{\mathcal{E}}^{\otimes d-1} \otimes {\mathcal{O}}_{\Delta} \ar[d] \ar[r] & {\mathcal{E}}^{\otimes d-1} \otimes \mathcal{E} \otimes \mathcal{E} \otimes {\mathcal{Q}}^{-1} \ar[r] \ar[d] & \frac{{\mathcal{E}}^{\otimes d}}{\mathcal{M}} \otimes \mathcal{E} \otimes {\mathcal{Q}}^{-1} \ar[d] \\
{\mathcal{E}}^{\otimes d-1} \otimes {\mathcal{O}}_{\Delta} \ar[r] & {\mathcal{E}}^{\otimes d-1} \otimes \mathcal{E} \otimes {\mathcal{E}}^{*}  \ar[r] & \frac{{\mathcal{E}}^{\otimes d}}{\mathcal{M}} \otimes {\mathcal{E}}^{*}
}
\end{equation}
whose left vertical is the identity, whose right two verticals are isomorphisms induced by (\ref{eqn.nondegen}) and whose bottom left horizontal is induced by the unit map of the adjoint pair $(-\otimes \mathcal{E},-\otimes {\mathcal{E}}^{*})$.  Tensoring the bottom row of (\ref{eqn.nonnon}) on the right by $\mathcal{E}$ yields a commutative diagram
\begin{equation} \label{eqn.newest}
\xymatrix{
{\mathcal{E}}^{\otimes d-1} \otimes \mathcal{E} \ar[r] & {\mathcal{E}}^{\otimes d-1} \otimes \mathcal{E} \otimes {\mathcal{E}}^{*} \otimes \mathcal{E} \ar[r] \ar[d] & \frac{{\mathcal{E}}^{\otimes d}}{\mathcal{M}} \otimes {\mathcal{E}}^{*} \otimes \mathcal{E} \ar[d] \\
 & {\mathcal{E}}^{\otimes d-1} \otimes \mathcal{E} \ar[r] & \frac{{\mathcal{E}}^{\otimes d}}{\mathcal{M}}
}
\end{equation}
whose verticals are induced by the counit map of the adjoint pair $(-\otimes \mathcal{E},-\otimes {\mathcal{E}}^{*})$.  Since both routes of (\ref{eqn.nonnon}) are $0$, and since the upper-left horizontal composed with the left vertical of (\ref{eqn.newest}) is the identity (Corollary \ref{cor.dualityprime}), the bottom horizontal must be zero.  This contradiction establishes the fact that $(Q)_{d+1}$ is not contained in $\mathcal{M}\otimes \mathcal{E}$, and the proof follows.
\end{proof}

\begin{proposition} \label{theorem.converge}
The point modules of $\mathcal{B}$ are parameterized by a closed subscheme of ${\mathbb{P}}_{X^{2}}(\mathcal{E})$ whose closed points equal those of ${\mathbb{P}}_{X^{2}}(\mathcal{E})$.
\end{proposition}

\begin{proof}
By Propositions \ref{prop.smith} and \ref{prop.linearfibre}, the hypothesis of Lemma \ref{lem.rulednew} are satisfied when we set $d=1$.  Since $\Gamma_{1}={\mathbb{P}}_{X^{2}}(\mathcal{E})$, the result follows.
\end{proof}

When ${\sf Proj}\mathcal{B}$ is a quantum ruled surface, Van den Bergh has proven \cite[Proposition 5.3.1]{14}, without the tools we develop in Section 2, that truncating a truncated family of point modules of length $i+1$ by taking its first $i$ components defines an isomorphism $\Gamma_{i} \rightarrow \Gamma_{i-1}$ of functors.  It follows immediately that, for $n \geq 1$ and $0 \leq i < j \leq n$, the morphisms $p_{i,j}^{n}$ are isomorphisms.  Thus, the point modules over $\mathcal{B}$ are parameterized by $\Gamma_{2}$, which is the graph of an automorphism of ${\mathbb{P}}_{X^{2}}(\mathcal{E})$.

\subsection{Resolutions of point modules}
Suppose $\mathcal{M}$ is a $\mathcal{B}$-point module, ${\mathcal{M}}_{-1}$ is the kernel of the multiplication map $\nu:{\mathcal{M}}_{0} \otimes \mathcal{E} \rightarrow {\mathcal{M}}_{1}$ and $\mu:\mathcal{E} \otimes {\mathcal{B}}_{i-2} \rightarrow {\mathcal{B}}_{i-1}$ is the multiplication morphism.

\begin{lemma} \label{lem.newfirst}
Let $X$ and $Y$ be schemes and suppose
$$
\xymatrix{
0 \ar[r] & \mathcal{K} \ar[r]^{\kappa} & \mathcal{M} \ar[r]^{\phi} & {\mathcal{M}}' \ar[r] & 0
}
$$
is a short exact sequence of ${\mathcal{O}}_{X}$-modules, and
$$
\xymatrix{
0 \ar[r] & \mathcal{N} \ar[r]^{\gamma} & \mathcal{F} \ar[r]^{\psi} & {\mathcal{F}}' \ar[r] & 0
}
$$
is a short exact sequence of ${\mathcal{O}}_{X}-{\mathcal{O}}_{Y}$-bimodules.  Then
$$
\operatorname{ker }\phi \otimes_{{\mathcal{O}}_{X}} \psi = \operatorname{im} \kappa \otimes_{{\mathcal{O}}_{X}} id_{\mathcal{F}} + \operatorname{im} id_{\mathcal{M}} \otimes_{{\mathcal{O}}_{X}} \gamma.
$$
\end{lemma}

\begin{proof}
The proof is nearly identical to the proof of \cite[Corollary 3.18, p.38]{8} so we omit it.
\end{proof}

\begin{lemma} \label{lem.newsecond}
If $i \geq 2$ then the pullback, $\mathcal{P}$, of the diagram
$$
\xymatrix{
& {\mathcal{M}}_{-1} \otimes \mathcal{E} \otimes {\mathcal{B}}_{i-2} \ar[d] \\
{\mathcal{M}}_{0} \otimes \mathcal{Q} \otimes {\mathcal{B}}_{i-2} \ar[r] & {\mathcal{M}}_{0} \otimes \mathcal{E} \otimes \mathcal{E} \otimes {\mathcal{B}}_{i-2}
}
$$
is trivial.
\end{lemma}

\begin{proof}
Consider the diagram
$$
\xymatrix{
& 0 \ar[d] \\
\mathcal{P} \ar[d] \ar[r] & {\mathcal{M}}_{-1} \otimes \mathcal{E} \otimes {\mathcal{B}}_{i-2} \ar[d] \\
{\mathcal{M}}_{0} \otimes \mathcal{Q} \otimes {\mathcal{B}}_{i-2} \ar[r] \ar[dr] & {\mathcal{M}}_{0} \otimes \mathcal{E} \otimes \mathcal{E} \otimes {\mathcal{B}}_{i-2} \ar[d] \\
& {\mathcal{M}}_{1} \otimes \mathcal{E} \otimes {\mathcal{B}}_{i-2} \ar[d] \\
& 0
}
$$
whose diagonal is induced by the monomorphism (\ref{eqn.immonic}).  Composing the left vertical with the diagonal gives a monomorphism ${\mathcal{P}} \rightarrow {\mathcal{M}}_{1} \otimes \mathcal{E} \otimes {\mathcal{B}}_{i-2}$.  On the other hand, since this diagram commutes and the center column is a short exact sequence, $\mathcal{P}$ must have trivial image.  Thus $\mathcal{P}$ is trivial as desired.
\end{proof}

\begin{lemma} \label{lem.newthird}
The pullback of the diagram
$$
\xymatrix{
& {\mathcal{M}}_{0} \otimes \mathcal{Q} \otimes {\mathcal{B}}_{i-2} \ar[d] \\
{\mathcal{M}}_{0} \otimes \mathcal{E} \otimes \operatorname{ker }\mu \ar[r] & {\mathcal{M}}_{0} \otimes \mathcal{E} \otimes \mathcal{E} \otimes {\mathcal{B}}_{i-2}
}
$$
is trivial.
\end{lemma}

\begin{proof}
To prove the result, it suffices to prove that the pullback of the diagram
$$
\xymatrix{
& \mathcal{Q} \otimes {\mathcal{B}}_{i-2} \ar[d] \\
\mathcal{E} \otimes \operatorname{ker }\mu \ar[r] & \mathcal{E} \otimes \mathcal{E} \otimes {\mathcal{B}}_{i-2}
}
$$
is trivial.  This holds if the pullback of
$$
\xymatrix{
& \frac{\mathcal{Q} \otimes {\mathcal{E}}^{\otimes i-2} + \mathcal{E} \otimes \mathcal{E} \otimes (\mathcal{Q})_{i-2}}{\mathcal{E} \otimes \mathcal{E} \otimes (\mathcal{Q})_{i-2}} \ar[d] \\
\frac{\mathcal{E} \otimes \mathcal{E} \otimes (\mathcal{Q})_{i-2}+\mathcal{E} \otimes (\mathcal{Q})_{i-2} \otimes \mathcal{E}}{\mathcal{E} \otimes \mathcal{E} \otimes (\mathcal{Q})_{i-2}} \ar[r] & \frac{\mathcal{E} \otimes \mathcal{E} \otimes {\mathcal{E}}^{\otimes i-2}}{\mathcal{E} \otimes \mathcal{E} \otimes (\mathcal{Q})_{i-2}}
}
$$
is trivial, which in turn follows from the triviality of the pullback of
$$
\xymatrix{
& \mathcal{Q} \otimes {\mathcal{E}}^{\otimes i-2} \ar[d] \\
\mathcal{E} \otimes \mathcal{Q} \otimes {\mathcal{E}}^{\otimes i-3} \ar[r] & {\mathcal{E}}^{\otimes i}.
}
$$
Finally, the pullback of this diagram is trivial if the pullback, $\mathcal{P}$, of
$$
\xymatrix{
& \mathcal{Q} \otimes \mathcal{E} \ar[d] \\
\mathcal{E} \otimes \mathcal{Q} \ar[r] & {\mathcal{E}}^{\otimes 3}
}
$$
is trivial.  Since
$$
0 \rightarrow \mathcal{Q} \otimes \mathcal{E} + \mathcal{E} \otimes \mathcal{Q} \rightarrow {\mathcal{E}}^{\otimes 3} \rightarrow {\mathcal{B}}_{3} \rightarrow 0
$$
is exact and ${\mathcal{B}}_{3}$ is locally free of rank four by Theorem \ref{theorem.homological}, $\mathcal{Q} \otimes \mathcal{E} + \mathcal{E} \otimes \mathcal{Q}$ is locally free of rank four.  Since
$$
0 \rightarrow \mathcal{P} \rightarrow \mathcal{Q} \otimes \mathcal{E} \oplus \mathcal{E} \otimes \mathcal{Q} \rightarrow \mathcal{Q} \otimes \mathcal{E} + \mathcal{E} \otimes \mathcal{Q} \rightarrow 0
$$
is exact, $\mathcal{P}$ is locally free of rank zero.  Thus $\mathcal{P}$ is trivial, as desired.
\end{proof}

\begin{proposition} \label{prop.newfirst}
The pullback of the diagram
$$
\xymatrix{
& {\mathcal{M}}_{-1} \otimes {\mathcal{B}}_{i-1} \ar[d] \\
{\mathcal{M}}_{0} \otimes \mathcal{Q} \otimes {\mathcal{B}}_{i-2} \ar[r] & {\mathcal{M}}_{0} \otimes \mathcal{E} \otimes {\mathcal{B}}_{i-1}
}
$$
whose horizontal is induced by the composition
$$
\xymatrix{
\mathcal{Q} \otimes {\mathcal{B}}_{i-2} \ar[r] & \mathcal{E} \otimes \mathcal{E} \otimes {\mathcal{B}}_{i-2} \ar[r]^{\mathcal{E} \otimes \mu} & \mathcal{E} \otimes {\mathcal{B}}_{i-1}
}
$$
is trivial.
\end{proposition}

\begin{proof}
Since the diagram
$$
\xymatrix{
& {\mathcal{M}}_{0}\otimes \mathcal{Q} \otimes {\mathcal{B}}_{i-2} \ar[r] & {\mathcal{M}}_{0}\otimes \mathcal{E} \otimes \mathcal{E} \otimes {\mathcal{B}}_{i-2} \ar[r] \ar[d] & {\mathcal{M}}_{1} \otimes \mathcal{E} \otimes {\mathcal{B}}_{i-2} \ar[d] \\
0 \ar[r] & {\mathcal{M}}_{-1} \otimes {\mathcal{B}}_{i-1} \ar[r] & {\mathcal{M}}_{0} \otimes \mathcal{E} \otimes {\mathcal{B}}_{i-1} \ar[r] & {\mathcal{M}}_{1} \otimes {\mathcal{B}}_{i-1} \ar[r] & 0
}
$$
whose bottom row is exact, commutes by the associativity of $\mathcal{B}$-module multiplication, it suffices to show that the top route is monic.  Thus, we must show that the pullback of the diagram
\begin{equation} \label{eqn.pullback}
\xymatrix{
& {\mathcal{M}}_{0} \otimes \mathcal{Q} \otimes {\mathcal{B}}_{i-2} \ar[d] \\
\operatorname{ker} \nu \otimes \mu \ar[r] & {\mathcal{M}}_{0} \otimes \mathcal{E} \otimes \mathcal{E} \otimes {\mathcal{B}}_{i-2}
}
\end{equation}
is trivial.  By Lemma \ref{lem.newfirst},
$$
\operatorname{ker} \nu \otimes \mu = {\mathcal{M}}_{-1} \otimes \mathcal{E} \otimes {\mathcal{B}}_{i-2}+{\mathcal{M}}_{0} \otimes \mathcal{E} \otimes \operatorname{ker } \mu.
$$
By Lemma \ref{lem.newsecond} and Lemma \ref{lem.newthird}, the pullback of (\ref{eqn.pullback}) is indeed trivial.
\end{proof}
Following \cite[Section 4]{7} we define the Hilbert series of a sequence of ${\mathcal{O}}_{X}$-modules in terms of the $K$-theory of $X$.  If $\mathcal{M}$ is a coherent ${\mathcal{O}}_{X}$-module, we let $[\mathcal{M}]$ denote the class of $\mathcal{M}$ in $K_{0}(X)$.

\begin{definition}
Suppose $\mathcal{M} = \{ {\mathcal{M}}_{i} \}_{i \in \mathbb{Z}}$ is a sequence of coherent ${\mathcal{O}}_{X}$-modules. The {\bf Hilbert series of } $\mathcal{M}$, denoted $H_{\mathcal{M}}(t)$ is the element $\Sigma_{i \in \mathbb{Z}}[{\mathcal{M}}_{i}]t^{i}$ of the ring $K_{0}(X)[t,t^{-1}]$.
\end{definition}
Now, assume $X={\mathbb{P}}^{1}$.

\begin{lemma} \label{cor.ktheory}
If $p \in X$ is a closed point, $[{\mathcal{O}}_{p} \otimes {\mathcal{B}}_{i}] = (i+1)[{\mathcal{O}}_{p}]$ and the Hilbert series of ${\mathcal{O}}_{p} \otimes_{{\mathcal{O}}_{X}} \mathcal{B}$ is
$$
([{\mathcal{O}}_{p}]-[{\mathcal{O}}_{p}]t+[{\mathcal{O}}_{p}]t^{2})^{-1}.
$$
\end{lemma}

\begin{proof}
This is an immediate consequence of Lemma \ref{lem.ktheory}.
\end{proof}

\begin{proposition} \label{lem.hseries}
An object $\mathcal{M}$ in ${\sf{Grmod }}\mathcal{B}$ with multiplication map $\rho:{\mathcal{M}}_{0}\otimes \mathcal{B} \rightarrow \mathcal{M}$ and isomorphism $\phi:{\mathcal{O}}_{p} \rightarrow {\mathcal{M}}_{0}$ for $p$ a closed point in $X$ has a graded ${\mathcal{O}}_{X}-{\mathcal{B}}$-bimodule resolution
\begin{equation} \label{eqn.ses}
\xymatrix{
0 \ar[r] & ({\mathcal{O}}_{q} \otimes_{{\mathcal{O}}_{X}} \mathcal{B})(-1) \ar[r] & {\mathcal{O}}_{p} \otimes_{{\mathcal{O}}_{X}} \mathcal{B} \ar[rr]^{\rho (\phi \otimes \mathcal{B})} & & {\mathcal{M}} \ar[r] & 0
}
\end{equation}
for $q$ a closed point in $X$ if and only if $\mathcal{M}$ is a point module.
\end{proposition}

\begin{proof}
Suppose $\mathcal{M}$ has a resolution (\ref{eqn.ses}).  From the $i$th component of (\ref{eqn.ses}), we find
$$
[{\mathcal{O}}_{p} \otimes {\mathcal{B}}_{i}]=[{\mathcal{M}}_{i}]+[{\mathcal{O}}_{q} \otimes {\mathcal{B}}_{i-1}].
$$
Multiplying this equation by $t^{i}$ and summing over $i$, we have
$$
H_{{\mathcal{O}}_{p}\otimes \mathcal{B}}(t)=H_{\mathcal{M}}(t)+H_{{\mathcal{O}}_{q}\otimes \mathcal{B}}(t)t.
$$
Since, $[{\mathcal{O}}_{p}]=[{\mathcal{O}}_{q}]$ over ${\mathbb{P}}^{1}$, we may conclude that
\begin{equation} \label{eqn.hilbert}
H_{\mathcal{M}}(t)=H_{{\mathcal{O}}_{p}\otimes \mathcal{B}}(t)([{\mathcal{O}}_{X}]-[{\mathcal{O}}_{X}]t).
\end{equation}
By Corollary \ref{cor.ktheory},
$$
H_{{\mathcal{O}}_{p}\otimes \mathcal{B}}(t)=([{\mathcal{O}}_{p}]-[{\mathcal{O}}_{p}]t+[{\mathcal{O}}_{p}]t^{2})^{-1}=([{\mathcal{O}}_{X}]-[{\mathcal{O}}_{X}]t)^{-1}([{\mathcal{O}}_{p}]-[{\mathcal{O}}_{p}]t)^{-1}.
$$
This, together with (\ref{eqn.hilbert}), implies that $H_{\mathcal{M}}(t)=([{\mathcal{O}}_{p}]-[{\mathcal{O}}_{p}]t)^{-1}$.

Conversely, suppose $\mathcal{M}$ is a point module, retain the notation as in the beginning of the section.  For $i \geq 2$, there is a commutative diagram whose top row is an exact sequence of ${\mathcal{O}}_{X}$-modules
$$
\xymatrix{
0 \ar[r] & {\mathcal{M}}_{0} \otimes \mathcal{Q} \otimes {\mathcal{B}}_{i-2} \ar[r] & {\mathcal{M}}_{0} \otimes \mathcal{E} \otimes {\mathcal{B}}_{i-1} \ar[r] & {\mathcal{M}}_{0} \otimes {\mathcal{B}}_{i} \ar[r] & 0 \\
& & {\mathcal{M}}_{-1} \otimes {\mathcal{B}}_{i-1} \ar[u] \ar[ur]_{\psi} & &
}
$$
by Theorem \ref{theorem.homological}.  Furthermore, $\operatorname{im }\psi \subset \operatorname{ker }\nu$ since
$$
\xymatrix{
0 \ar[r] & {\mathcal{M}}_{-1}\otimes {\mathcal{B}}_{i-1} \ar[r] & {\mathcal{M}}_{0} \otimes \mathcal{E} \otimes {\mathcal{B}}_{i-1} \ar[r] \ar[d] & {\mathcal{M}}_{1} \otimes {\mathcal{B}}_{i-1} \ar[r] \ar[d] & 0\\
& & {\mathcal{M}}_{0} \otimes {\mathcal{B}}_{i} \ar[r] & {\mathcal{M}}_{i}
}
$$
is commutative.  If $\psi$ were monic ${\mathcal{M}}_{0} \otimes {\mathcal{B}}_{i}/\operatorname{im} \psi$ would have length 1 so that since ${\mathcal{M}}_{i}$ is a quotient of ${\mathcal{M}}_{0} \otimes {\mathcal{B}}_{i}/\operatorname{im} \psi$, ${\mathcal{M}}_{i}$ would be isomorphic to ${\mathcal{M}}_{0} \otimes {\mathcal{B}}_{i}/\operatorname{im} \psi$.  But, since $\operatorname{ker }\psi$ is the pullback of the diagram
$$
\xymatrix{
& {\mathcal{M}}_{-1}\otimes {\mathcal{B}}_{i-1} \ar[d] \\
{\mathcal{M}}_{0} \otimes \mathcal{Q} \otimes {\mathcal{B}}_{i-2} \ar[r] & {\mathcal{M}}_{0} \otimes \mathcal{E} \otimes {\mathcal{B}}_{i-1},
}
$$
$\psi$ must be monic by Proposition \ref{prop.newfirst}.
\end{proof}

\end{document}